\documentclass[10pt]{amsart}

\usepackage{amscd} 

\newtheorem{thm}{Theorem}[section]
\newtheorem{lem}[thm]{Lemma}
\newtheorem{prop}[thm]{Proposition}
\newtheorem{cor}[thm]{Corollary}

\theoremstyle{definition}

\newtheorem{example}[thm]{Example}

\theoremstyle{rem}
\newtheorem{rem}[thm]{Remark}

\numberwithin{equation}{section}

\begin{document}

\title[B\"{o}ttcher coordinates]
{A construction of B\"{o}ttcher coordinates for holomorphic skew products}  

\author[K. Ueno]{Kohei Ueno}
\address{Daido University, Nagoya 457-8530, Japan}
\curraddr{}
\email{k-ueno@daido-it.ac.jp}
\thanks{This work was supported by the Research Institute for Mathematical Sciences, 
           a Joint Usage/Research Center located in Kyoto University.}


\date{}
\dedicatory{}

\begin{abstract}
Let $f(z,w)=(p(z),q(z,w))$ be a holomorphic skew product
with a superattracting fixed point at the origin.
Under one or two assumptions,
we prove that $f$ is conjugate to a monomial map
on an invariant open set whose closure contains the origin.
The monomial map and the open set are determined by
the degree of $p$ and the Newton polygon of $q$.
\end{abstract}

\maketitle

\section{Introduction}

Let $p : (\mathbb{C}, 0) \to (\mathbb{C}, 0)$ be 
a holomorphic germ with a superattracting fixed point at the origin.
We may write $p(z) = a_{\delta} z^{\delta} + O(z^{\delta + 1})$,
where $a_{\delta} \neq 0$ and $\delta \geq 2$.
Let $p_0(z) = a_{\delta} z^{\delta}$.
B\"{o}ttcher's theorem \cite{b} provides 
a conformal function $\varphi$ defined on a neighborhood of the origin, 
with $\varphi (0) = 0$ and $\varphi' (0) = 1$, 
that conjugates $p$ to $p_0$.
This function is called the B\"{o}ttcher coordinate for $p$ at the origin
and obtained as the limit of 
the compositions of $p_0^{-n}$ and $p^n$,
where $p^n$ denotes the $n$-th iterate of $p$.
The branch of $p_0^{-n}$ is taken such that $p_0^{-n} \circ p_0^n = id$.
We refer to \cite{m} for details.

Several studies have been made
toward the generalization of B\"{o}ttcher's theorem to higher dimensions.
For example,
Ushiki~\cite{ushiki}, Ueda~\cite{ueda},
Buff, Epstein and Koch~\cite{bek} studied the case in which holomorphic germs,
with superattracting fixed points, have the B\"{o}ttcher coordinates
on neighborhoods of the points.
The germs in \cite{ushiki} are conjugate to monomial maps, 
whereas the germs in \cite{ueda} and \cite{bek} are conjugate to 
homogeneous and quasihomogeneous maps,
respectively.

However, 
B\"{o}ttcher's theorem does not extend to higher dimensions entirely
as pointed out by Hubbard and Papadopol \cite{hp}. 
If two germs are conjugate,
then the two critical orbits should be preserved by the conjugacy.
Although the critical orbit of a normal form is expected to be simple,
that of a given germ is usually very complicated.

Rigidity is a keyword for the study of the local dynamics of superattracting germs.
Favre \cite{f} classified attracting rigid germs in dimension 2;
a germ is called rigid if the union of the critical sets of all its iterates is 
a divisor with normal crossing and forward invariant.
Favre and Jonsson \cite{fj} have built up a general theorem: 
for any superattracting germ in dimension $2$,
it can be blown up to a rigid germ with a fixed point on the exceptional divisor.
See Theorems C and 5.1 in \cite{fj} for details.
Therefore,
the original germ is conjugate to a normal form
on an open set whose closure contains the superattracting fixed point.
One can also find this theorem in a survey article \cite{a}
on the local dynamics
of holomorphic germs with fixed points of several types 
in one and higher dimensions.

In this paper
we deal with holomorphic skew products
with superattracting fixed points at the origin,
and construct B\"{o}ttcher coordinates 
on invariant open sets whose closure contain the origin.
This is a continuation of our studies \cite{u1, u2, u3},
and gives a well organized consequence that includes the main results 
for the skew product and superattracting case in \cite{u2, u3}.
Moreover,
we obtain statements on the uniqueness and extension of the B\"{o}ttcher coordinates,
which are similar to those in \cite{u1}.

For the study of the (global) dynamics of polynomial skew products,
we refer to \cite{fg, j},
in which the main topics are the Green functions, currents and measures. 
Lilov \cite{l} studied the local and semi-local dynamics of 
holomorphic skew products near a superattracting invariant fiber.
As natural extensions of the one dimensional results,
he obtained nice normal forms on neighborhoods of periodic points
which are geometrically attracting, parabolic and Siegel on fiber direction,
except the superattracting case.
See also \cite{abdpr, bfp, pr, ps, pv}
for the dynamics of skew products near an invariant fiber
of different types.
In particular,
wandering Fatou components are constructed in \cite{abdpr}
for polynomial skew products near a parabolic fiber.

Let us state our main results precisely.
Let $f : (\mathbb{C}^2, 0) \to (\mathbb{C}^2, 0)$ be 
a holomorphic germ of the form $f(z,w)=(p(z),q(z,w))$,
which is called a holomorphic skew product in this talk.
We assume that it has a superattracting fixed point at the origin; 
that is, $f(0) = 0$ and the eigenvalues of $Df(0)$ are both zero.
Then we may write $p(z) = a_{\delta} z^{\delta} + O(z^{\delta + 1})$,
where $a_{\delta} \neq 0$ and $\delta \geq 2$, and
$q(z,w) = b z + \sum_{i, j \geq 0, i+j \geq 2} b_{ij} z^{i} w^{j}$.
Let $bz = b_{10} z^1 w^0$ and $q(z,w) = \sum b_{ij} z^{i} w^{j}$ for short.
It is clear that the dominant term of $p$ is $a_{\delta} z^{\delta}$.
On the other hand,
there is a ``dominant'' term $b_{\gamma d} z^{\gamma} w^{d}$ of $q$ 
determined by the degree of $p$ and the Newton polygon of $q$; thus
\[
p(z) = a_{\delta} z^{\delta} + O(z^{\delta + 1})
\text{ and }
q(z,w) = b_{\gamma d} z^{\gamma} w^{d} + \text{$\sum_{(i,j) \neq (\gamma, d)}$} b_{ij} z^{i} w^{j}.
\]
More precisely,
$b_{\gamma d} z^{\gamma} w^{d}$ is dominant on an open set
\[
U = U_r = \{ |z|^{l_1 + l_2} < r^{l_2} |w|, |w| < r|z|^{l_1} \}
\]
for some rational numbers $0 \leq l_1 < \infty$ and $0 < l_2 \leq \infty$,
which are also determined by the degree of $p$ and the Newton polygon of $q$.

Let $f_0 (z,w) = (a_{\delta} z^{\delta}, b_{\gamma d} z^{\gamma} w^{d})$
and $||(z,w)|| = \max \{ |z|, |w| \}$.

\begin{lem}\label{main lemma}
If $d \geq 2$, 
then
\begin{enumerate}
\item for any small $\varepsilon >0$ 
there is $r >0$ such that
$||f - f_0|| < \varepsilon ||f_0||$ on $U_r$,
and 
\item $f(U_r) \subset U_r$ for small $r >0$.
\end{enumerate}
\end{lem}

In particular,
$f$ is rigid on $U_r$.
As in the one dimensional case,
this lemma induces a conjugacy on $U_r$ 
from $f$ to the monomial map $f_0$.

\begin{thm}\label{main theorem}
If $d \geq 2$, then
there is a biholomorphic map $\phi$ defined on $U_r$
that conjugates $f$ to $f_0$
for small $r > 0$.
Moreover,
for any small $\varepsilon > 0$, there is $r > 0$ such that
$||\phi - id|| < \varepsilon ||id||$ on $U_r$.
\end{thm}

We call $\phi$ the B\"{o}ttcher coordinate for $f$ on $U$,
and construct it as the limit of 
the compositions of $f_0^{-n}$ and $f^n$.

Let us give the definition of the Newton polygon of $q$,
and explain how the dominant term $b_{\gamma d} z^{\gamma} w^{d}$ and 
the rational numbers $l_1$ and $l_2$ are determined.
Let $q(z,w) = \sum b_{ij} z^{i} w^{j}$.
We define the Newton polygon $N(q)$ of $q$ as 
the convex hull of the union of $D(i,j)$ with $b_{ij} \neq 0$,
where $D(i,j) = \{ (x,y) \ | \ x \geq i, y \geq j \}$. 
Let $(n_1, m_1)$, $(n_2, m_2)$, $\cdots, (n_s, m_s)$ be the vertices of $N(q)$,
where $n_1 < n_2 < \cdots < n_s$ and $m_1 > m_2 > \cdots > m_s$.
Let $T_k$
be the y-intercept of the line $L_k$ passing the vertices $(n_k, m_k)$ and $(n_{k+1}, m_{k+1})$
for each $1 \leq k \leq s-1$.
Note that
\[
U = \{ |z|^{l_1 l_2^{-1} + 1} < r |w|^{l_2^{-1}}, |w| < r|z|^{l_1} \},
\]
where $0 \leq l_1 < \infty$ and $0 \leq l_2^{-1} < \infty$,
and so $U = \{ |z| < r, |w| < r |z|^{l_1} \}$ if $l_2^{-1} = 0$.

\vspace{2mm}
\begin{itemize}
\item[\underline{Case 1}] 
If $s = 1$, then $N(q)$ has the only one vertex, which is denoted by $(\gamma, d)$. \\
For this case, we define
$l_1 =  l_2^{-1} = 0$ and so $U = \{ |z| < r, |w| < r \}$.
\end{itemize}
\vspace{2mm}

For Case 1, 
$b_{\gamma d} z^{\gamma} w^{d}$ is clearly the dominant term of $q$
since $\gamma \leq i$ and $d \leq j$ for any $i$ and $j$ such that $b_{ij} \neq 0$,
and the results are classical.

Difficulties appear when $s > 1$,
which is divided into the following three cases. 

\vspace{2mm}
\begin{itemize}
\item[\underline{Case 2}] 
  If $s > 1$ and $\delta \leq T_{s-1}$, then  we define
  \[
  (\gamma, d) = (n_s, m_s), \ 
  l_1 = \frac{n_s - n_{s-1}}{m_{s-1} - m_s} 
  \text{ and } l_2^{-1} = 0.
  \] 
  Hence $U = \{ |z| < r, |w| < r |z|^{l_1} \}$. 
  \vspace{4mm}
\item[\underline{Case 3}] 
  If $s > 1$ and $T_1 \leq \delta$, then  we define
  \[
  (\gamma, d) = (n_1, m_1), \
  l_1 = 0 
  \text{ and } l_2 = \frac{n_2 - n_1}{m_1 - m_2}.
  \] 
  Hence $U = \{ |z|^{l_2} < r^{l_2} |w|, |w| < r \} = \{ r^{-l_2} |z|^{l_2} < |w| < r \}$.  
  \vspace{4mm}
\item[\underline{Case 4}]
  If $s > 1$ and $T_k \leq \delta \leq T_{k-1}$ for some $2 \leq k \leq s-1$, then  we define
  \[
  (\gamma, d) = (n_k, m_k), \ 
  l_1 = \frac{n_k - n_{k-1}}{m_{k-1} - m_k} 
  \text{ and } l_1 + l_2= \frac{n_{k+1} - n_k}{m_k - m_{k+1}}.
  \] 
  Hence  $U = \{ |z|^{l_1 + l_2} < r^{l_2} |w|, |w| < r|z|^{l_1} \} = \{ r^{-l_2} |z|^{l_1 + l_2} < |w| < r|z|^{l_1} \}$. 
\end{itemize}
\vspace{2mm}

\begin{rem}[Slope of $L_k$]
For Case 4,
the rational numbers
$- l_1^{-1}$ and $- (l_1 + l_2)^{-1}$ are the slopes of the lines $L_{k-1}$ and $L_k$.
The same correspondence holds for all cases
if we define $L_0 = \{ x = n_1 \}$ and $L_s = \{ y = m_s \}$.
\end{rem}

\begin{rem}[Two dominant terms]
If $s > 1$ and $\delta = T_k$ for some $1 \leq k \leq s-1$,
then there are two different ``dominant'' terms of $q$.
Moreover,
if both satisfy the degree condition,
then there are two disjoint invariant open sets 
on which $f$ is conjugate to each of the two different monomial maps.
\end{rem}

\begin{rem}[Comparision with our previous results]
The results for Cases 2 and 3 were already proved in \cite{u3} and \cite{u2}, respectively.
In this paper we succeed in solving Case 4
and giving a unified statement for all cases in terms of the Newton polygon.
\end{rem}

\begin{rem}[Extension of $\phi$]
Using similar arguments in \cite{u1},
we prove that $\phi$ extends by analytic continuation
until it meets the other critical set of $f$ than the $z$-axis and $w$-axis in Section 9.
On the other hand,
if $m_j \geq 2$ for any $j$, then
$\phi$ does not extend 
from $U$ to a neighborhood of the origin
for Cases 2, 3 and 4,
because the critical set of $f$ consists not only of $\{ zw = 0 \}$ 
but also of other curves passing the origin.
\end{rem}

The same results hold even if $d = 1$
with one additional condition.

\begin{lem}\label{main lemma for d=1}
If $d = 1$ and $\delta \neq T_k$ for any $k$, 
then 
\begin{enumerate}
\item
for any small $\varepsilon>0$
there is $r>0$ such that
$||f - f_0|| < \varepsilon ||f_0||$ on $U_r$, 
and 
\item $f(U_r) \subset U_r$ for small $r>0$.
\end{enumerate}
\end{lem}

\begin{thm}\label{main theorem for d=1}
If $d = 1$ and $\delta \neq T_k$ for any $k$, then
there is a biholomorphic map $\phi$ defined on $U_r$
that conjugates $f$ to $f_0$
for small $r > 0$.
Moreover,
for any small $\varepsilon>0$, there is $r>0$ such that
$||\phi - id|| < \varepsilon ||id||$ on $U_r$.
\end{thm}

We can not remove the additional condition $\delta \neq T_k$
as stated in Example \ref{example}.
This condition is always satisfied for Case 1
since we have no $T_k$.
As stated in \cite{u2}, 
if $f$ is in Case 3 and $d=1$, 
then it is rigid of class 4 in \cite{f}
and conjugate to $f_0$ on a neighborhood of the origin,
not only on the wedge.

The organization of this paper is as follows.
The main purpose of this paper is to prove Lemma \ref{main lemma},
and we prove it for Cases 2, 3 and 4 in Sections 2, 3 and 4, respectively.
Although Cases 2 and 3 were already proved in \cite{u2, u3},
we provide unified explanations in terms of Newton polygons and blow-ups, and
Case 4 is proved by combining arguments in Cases 2 and 3.
We omit the proofs of the main lemmas and theorems for Case 1;
the proofs are similar to and simpler than the other cases,
or one may refer to \cite{f}. 

In Section 5 
we introduce intervals of real numbers for each of which Lemma \ref{main lemma} holds.
Moreover,
we associate rational numbers in the intervals to branched coverings of $f$,
which are a generalization of the blow-ups,
and consider when the covering is well-defined. 
This section is a kind of an appendix,
and one may skip for the proofs of the main results.

Theorem \ref{main theorem} is proved in Section 6
by the same arguments as in \cite{u2}:
it follows from Lemma \ref{main lemma} that
the composition $\phi_n = f_0^{-n} \circ f^n$ is well-defined on $U_r$,
converges uniformly to $\phi$ on $U_r$,
and the limit $\phi$ is biholomorphic on $U_r$.
We use Rouch\'e's theorem 
to obtain the injectivity of $\phi$,
and one might need to shrink $r$ a little from that of the lemma.

The case $d = 1$ is dealt with in Section 7.
The proof of the uniform convergence of $\phi_n$ is different from the case $d \geq 2$,
and we prove that the same idea as in \cite{u2} works also for Case 4.
Example \ref{example} shows that we can not remove the additional condition.

The uniqueness and extension of $\phi$ is considered in Sections 8 and 9. 
In Section 8, using almost the same arguments as in \cite{u1},
we prove that a uniqueness statement similar to the one dimensional case holds
for Cases 1 and 2 with two suitable conditions if $d \geq 2$.
We deal with the extension problem of $\phi$ In Section 9.
Although the situation is different from that in \cite{u1},
almost the same arguments as in \cite{u1} work 
outside the $z$-axis and $w$-axis,
and we prove that 
$\phi$ extends by analytic continuation
until it meets the other critical set of $f$ than the $z$-axis and $w$-axis.

\section{Main lemma, Blow-ups and Newton polygons for Case 2}

We prove Lemma \ref{main lemma} for Case 2 in this section. 
Let
\[
\delta \leq T_{s-1}, \
(\gamma, d) = (n_s, m_s), \
l_1 = \frac{n_s - n_{s-1}}{m_{s-1} - m_s}
\text{ and } l_2^{-1} = 0.
\]
Recall that 
$p(z) = a_{\delta} z^{\delta} + O(z^{\delta + 1})$ and 
$q(z,w) = b_{\gamma d} z^{\gamma} w^{d} + \sum b_{ij} z^{i} w^{j}$.
By taking an affine conjugate,
we may assume that $a_{\delta} = 1$ and $b_{\gamma d} = 1$ if $d \geq 2$.
Moreover,
we may assume that $p(z) = z^{\delta}$.
In fact,
using the B\"{o}ttcher coordinate for the original $p$,
we can conjugate the original germ $f$ to a holomorphic skew product
whose first component is just $z^{\delta}$,
and the Newton polygons of the second components of the both germs are the same.
Therefore,
we may write
\[
f(z, w) = \left( z^{\delta}, z^{\gamma} w^{d} + \sum b_{ij} z^{i} w^{j} \right).
\]
Even if we do not impose these assumptions,
similar arguments in this paper induce the same results.

We first prove Lemma \ref{main lemma} in Section 2.1,
and then explain our results in terms of blow-ups when $l_1$ is integer 
and of Newton polygons in Sections 2.2 and 2.3.
Let us denote $f \sim f_0$ on $U_r$ as $r \to 0$ for short,
if $f$ satisfies the former statement in Lemma \ref{main lemma}:
for any small $\varepsilon$ 
there is $r$ such that
$||f - f_0|| < \varepsilon ||f_0||$ on $U_r$.

\subsection{Proof of the main lemma}

The following lemma is clear since $d = m_s$.

\begin{lem} \label{lem 1 for main lem: case2} 
It follows that $d \leq j$ for any $j$ such that $b_{ij} \neq 0$.
\end{lem}

More precisely,
$(\gamma, d)$ is minimum in the sense that
$d \leq j$, and $\gamma \leq i$ if $d = j$.

\begin{lem} \label{lem 2 for main lem: case2} 
It follows that
$l_1 \delta \leq \gamma + l_1 d \leq i + l_1 j$
for any $(i,j)$ such that $b_{ij} \neq 0$.
\end{lem}

\begin{proof}
These numbers $l_1 \delta$, $\gamma + l_1 d$ and $i + l_1 j$ are 
the $x$-intercepts of the lines with slope $-l_1^{-1}$
passing the points $(0, \delta)$, $(\gamma, d)$ and $(i,j)$.
\end{proof}

These inequalities in Lemmas \ref{lem 1 for main lem: case2} 
and \ref{lem 2 for main lem: case2}  induce the main lemma.

\begin{proof}[Proof of Lemma \ref{main lemma} for Case 2]  
We first define 
\[
\eta (z,w) = \frac{q(z,w) - z^{\gamma} w^d}{ z^{\gamma} w^d}
\]
and show the former statement.  
Let $l=l_1$ and $|w| = |z^l c|$.
Then 
\[
U_r = \{ |z| < r, |w| < r |z|^{l} \} = \{ 0 < |z| < r, |c| < r \} \subset \{ |z| < r, |c| < r \} \text{ and}
\] 
\[
|\eta|
= \left| \sum \frac{b_{ij} z^{i} w^{j}}{z^{\gamma} w^{d}} \right| 
= \left| \sum \frac{b_{ij} z^{i} (z^l c)^{j}}{z^{\gamma} (z^l c)^{d}} \right| 
= \left| \sum \frac{b_{ij} z^{i +lj} c^{j}}{z^{\gamma + ld} c^{d}} \right|
\]
\[
\leq \sum |b_{ij}| |z|^{(i + lj) - (\gamma + ld)} |c|^{j - d}.
\]
The conditions $i + lj \geq \gamma + ld$ and $j \geq d$ ensure that
the left-hand side is a power series in $|z|$ and $|c|$, 
and so converges on $\{ |z| < r, |c| < r \}$.
Moreover, 
for each $(i,j) \neq (\gamma, d)$,
at least one of the inequalities $(i + lj) - (\gamma + ld) > 0$ and $j - d > 0$ holds
since $j \geq d$, and $i > \gamma$ if $j = d$.
More precisely,
$(i + lj) - (\gamma + ld) \geq 1$ and/or $j - d \geq 1$.
Therefore, 
for any small $\varepsilon$
there is $r$ such that
$|\eta| < \varepsilon$ on $U_r$. 

We next show the invariance of $U_r$.
Since the inequality $|p(z)| < r$ is trivial,
it is enough to show that
$|q(z,w)| < r |p(z)|^{l}$ for any $(z,w)$ in $U_r$.
Because $\gamma + ld \geq l \delta$,
\[
\left| \frac{q(z,w)}{p(z)^{l}} \right| 
\sim \left| \frac{z^{\gamma} w^{d}}{(z^{\delta})^{l}} \right| 
= \left| \frac{z^{\gamma} (z^l c)^{d}}{(z^{\delta})^{l}} \right| 
= |z|^{\gamma + ld - l \delta}  |c|^{d}
\leq |c|^d < r^d
\]
on $U_r$ as $r \to 0$.
The condition $d \geq 2$ ensures that
$|q(z,w)/p(z)^{l}| \leq C r^d \leq C r^2 < r$ 
for some constant $C$ and sufficiently small $r$.
\end{proof}

\subsection{Blow-ups}

Assuming that $l_1$ is integer,
we explain our results in terms of blow-ups.
Let $\pi_1 (z,c) = (z, z^{l} c)$
and $\tilde{f} = \pi_1^{-1} \circ f \circ \pi_1$,
where $l = l_1$.
Note that $\pi_1$ is the $l$-th compositions of the blow-up $(z,c) \to (z,zc)$. 
Then we have
\[
\tilde{f} (z,c) 
= (\tilde{p}(z), \tilde{q}(z,c))
= \left( p(z), \ \dfrac{q(z,z^l c)}{p(z)^l} \right) 
\]
\[
= \left( z^{\delta}, \ 
z^{\gamma+ld - l \delta} c^d + \sum b_{ij} z^{i + lj - l \delta} c^{j} \right)
\]
\[
= \left( z^{\delta}, \ 
z^{\gamma+ld - l \delta} c^d
\left\{ 1 + \sum b_{ij} z^{(i + lj) - (\gamma + ld)} c^{j - d} \right\}  \right)
\]
\[
= \left( z^{\delta}, \ 
z^{\gamma+ld - l \delta} c^d
\left\{ 1 + \eta (z,c) \right\}  \right)
\sim ( z^{\delta}, \ z^{\gamma+ld - l \delta} c^d).
\]
Note that $\pi_1^{-1} (U_r) = \{ 0< |z| < r, |c| < r \} \subset \{ |z| < r, |c| < r \}$.

\begin{prop} \label{} 
If $l = l_1 \in \mathbb{N}$, 
then $\tilde{f}$ is well-defined, holomorphic, skew product and rigid 
on a neighborhood $\{ |z| < r, |c| < r \}$ of the origin. 
More precisely,
\[
\tilde{f} (z,c) = \left( z^{\delta}, \ 
z^{\gamma+ld - l \delta} c^d \left\{ 1 + \eta (z,c) \right\}  \right),
\]
where $\eta \to 0$ as $z$, $c \to 0$, and 
it has a superattracting fixed point at the origin.
\end{prop}

Because $\tilde{f}$ is a holomorphic skew product in Case 1,
it is easy to construct the B\"{o}ttcher coordinate for $\tilde{f}$,
which induces the B\"{o}ttcher coordinate for $f$ on $U_r$.

\begin{rem} \label{}  
If $l_1 \in \mathbb{N}$, 
then $\tilde{f}$ is well-defined 
not only on a neighborhood of the origin 
but also on the preimage of the domain of $f$ by $\pi_1$,
which includes the $c$-axis.
Moreover,
even if $l_1$ is rational,
we can lift $f$ to a holomorphic skew product similar to $\tilde{f}$
as stated in Proposition \ref{branched coverings: case2} in Section 5.1.
\end{rem}

\subsection{Newton polygons}

It is helpful to consider the Newton polygon of $\tilde{q}$.
Let 
\[
\tilde{\gamma} = \gamma + l_1 d - l_1 \delta
\text{ and }
\tilde{i} = i + l_1 j - l_1 \delta.
\]
Then 
$\tilde{q}(z,c) = z^{\tilde{\gamma}} c^d + \sum b_{ij} z^{\tilde{i}} c^{j}$
and Lemma \ref{lem 2 for main lem: case2} is translated into the following.

\begin{lem} \label{} 
It follows that
$0 \leq \tilde{\gamma} \leq \tilde{i}$ for any $(i,j)$ such that $b_{ij} \neq 0$.
\end{lem}

Therefore,
the Newton polygon of $\tilde{q}$ has just one vertex $(\tilde{\gamma}, d)$:
$N(\tilde{q}) = D(\tilde{\gamma}, d)$.

\begin{rem}
The affine transformation
\[
A_1
\begin{pmatrix} i \\ j \end{pmatrix} =
\begin{pmatrix} i + l_1 j - l_1 \delta \\ j \end{pmatrix} =
\begin{pmatrix} 1 & l_1 \\ 0 & 1 \end{pmatrix}
\begin{pmatrix} i \\ j \end{pmatrix} - \begin{pmatrix} l_1 \delta \\ 0 \end{pmatrix}
\]
maps the basis $\{ (1,0), (-l_1,1) \}$ to $\{ (1,0), (0,1) \}$.
In other words, 
$A_1$ maps a horizontal line and the line $L_{s-1}$ with slope $-l_1^{-1}$
to the same horizontal line and a vertical line.
\end{rem}

\section{Main lemma, Blow-ups and Newton polygons for Case 3}

We prove Lemma \ref{main lemma} for Case 3 in this section. 
Let
\[
T_1 \leq \delta, \
(\gamma, d) = (n_1, m_1), \
l_1 = 0 \text{ and }
l_2 = \frac{n_2 - n_1}{m_1 - m_2}. 
\]
Let us write $f(z, w) = \left( z^{\delta}, z^{\gamma} w^{d} + \sum b_{ij} z^{i} w^{j} \right)$
for simplicity.
Similar to the previous section,
we prove Lemma \ref{main lemma} in Section 3.1,
explain our results in terms of blow-ups when $l_2^{-1}$ is integer and of Newton polygons 
in Sections 3.2 and 3.3.

\subsection{Proof of the main lemma}

The following lemma is clear since $\gamma = n_1$.

\begin{lem} \label{lem1 for main lem: case3} 
It follows that
$\gamma \leq i$ for any $i$ such that $b_{ij} \neq 0$.
\end{lem}

More precisely,
$(\gamma, d)$ is minimum in the sense that
$\gamma \leq i$, and $d \leq j$ if $\gamma = i$.

\begin{lem} \label{lem2 for main lem: case3} 
It follows that
$l_2^{-1} \gamma + d \leq l_2^{-1} i + j$ and $l_2^{-1} \gamma + d \leq \delta$
for any $(i,j)$ such that $b_{ij} \neq 0$.
\end{lem}

\begin{proof}
The numbers $l_2^{-1} \gamma + d$ and $l_2^{-1} i + j$ are 
the $y$-intercepts of the lines with slope $-l_2^{-1}$
passing the points $(\gamma, d)$ and $(i,j)$.
In particular,
$l_2^{-1} \gamma + d = T_1 \leq \delta$.
\end{proof}

These inequalities in Lemmas \ref{lem1 for main lem: case3} 
and \ref{lem2 for main lem: case3}  induce the main lemma.

\begin{proof}[Proof of Lemma \ref{main lemma} for Case 3] 
We first define $\eta (z,w) = (q(z,w) - z^{\gamma} w^d)/ z^{\gamma} w^d$
and show the former statement.
Let $l=l_2^{-1}$ and $|z| = |tw^l|$.
Then 
\[
U_r = \{ |z| < r |w|^{l}, |w| < r \} = \{ |t| < r, 0 < |w| < r \} \subset \{ |t| < r, |w| < r \} \text{ and} 
\]
\[
|\eta|
= \left| \sum \frac{b_{ij} z^{i} w^{j}}{z^{\gamma} w^{d}} \right| 
= \left| \sum \frac{b_{ij} (tw^l)^{i} w^{j}}{(tw^l)^{\gamma} w^{d}} \right| 
= \left| \sum \frac{b_{ij} t^{i} w^{l i +j}}{t^{\gamma} w^{l \gamma + d}} \right|
\]
\[
\leq \sum |b_{ij}| |t|^{i - \gamma} |w|^{(l i + j) - (l \gamma + d)}.
\]
The conditions $i \geq \gamma$ and $l i + j \geq l \gamma + d$ ensure that
the left-hand side is a power series in $|t|$ and $|w|$, 
and so converges on $\{ |t| < r, |w| < r \}$.
Moreover, 
at least one of the inequalities $i > \gamma$ 
and $l i + j > l \gamma + d$ holds
since $i \geq \gamma$, and $j > d$ if $i = \gamma$.
More precisely,
$i - \gamma \geq 1$ and/or $(l i + j) - (l \gamma + d) \geq 1$.
Therefore, 
for any small $\varepsilon$ 
there is $r$ such that
$|\eta| < \varepsilon$ on $U_r$. 

We next show the invariance of $U_r$.
Since the inequality $|q(z,w)| < r$ is trivial,
it is enough to show that
$|p(z)| < r |q(z,w)|^{l}$ for any $(z,w)$ in $U_r$.
Because $\delta \geq l \gamma + d$,
\[
\left| \frac{p(z)}{q(z,w)^{l}} \right| 
\sim \left| \frac{z^{\delta}}{(z^{\gamma} w^{d})^{l}} \right| 
= \left| \frac{(tw^{l})^{\delta}}{((tw^{l})^{\gamma} w^{d})^{l}} \right|
= |t|^{\delta - l \gamma} |w|^{l \{ \delta - (l \gamma + d) \}}
\leq |t|^d < r^d
\]
on $U_r$ as $r \to 0$.
The condition $d \geq 2$ ensures that
$|p(z)/q(z,w)^{l}| \leq C r^d \leq C r^2 < r$ 
for some constant $C$ and sufficiently small $r$.
\end{proof}

\subsection{Blow-ups}

Assuming that $l_2^{-1}$ is integer,
we explain our results in terms of blow-ups.
Let $\pi_2 (t,w) = (t w^l, w)$ and $\tilde{f} = \pi_2^{-1} \circ f \circ \pi_2$,
where $l = l_2^{-1}$.  
Note that $\pi_2$ is the $l$-th compositions of the blow-up $(t,w) \to (tw,w)$.
Then we have
\[
\tilde{q} (t,w) = q(tw^l,w)
= t^{\gamma} w^{l \gamma + d} + \sum b_{ij} t^{i} w^{l i + j}
\]
\[
= t^{\gamma} w^{l \gamma + d} \left\{ 1 + \sum b_{ij} t^{i - \gamma} w^{(l i + j) - (l \gamma + d)} \right\}
= t^{\gamma} w^{l \gamma + d} \{ 1 + \eta (t,w) \}
\text{ and so}
\]
\[
\tilde{f} (t,w) 
= (\tilde{p} (t,w), \tilde{q} (t,w))
= \left( \dfrac{p(tw^l)}{q(tw^l,w)^l}, \ q(tw^l,w) \right)
\]
\[
= \left( \dfrac{t^{\delta - l \gamma} w^{l \{ \delta - (l \gamma + d) \}}}{\{ 1 + \eta (t,w) \}^{l}}, 
\ t^{\gamma} w^{l \gamma + d} \{ 1 + \eta (t,w) \} \right)
\sim (t^{\delta - l \gamma} w^{l \{ \delta - (l \gamma + d) \}}, \ t^{\gamma} w^{l \gamma + d}).
\]
Note that $\pi_2^{-1} (U_r) = \{ |t| < r, 0 < |w| < r \} \subset \{ |t| < r, |w| < r \}$.

\begin{prop} \label{} 
If $l = l_2^{-1} \in \mathbb{N}$, 
then $\tilde{f}$ is well-defined, holomorphic and rigid
on a neighborhood $\{ |t| < r, |w| < r \}$ of the origin. 
More precisely,
\[
\tilde{f} (t,w) = 
\left( t^{\delta - l \gamma} w^{l \{ \delta - (l \gamma + d) \}} \{ 1 + \zeta (t,w) \},
\ t^{\gamma} w^{l \gamma + d} \{ 1 + \eta (t,w) \} \right),
\]
where $\zeta$, $\eta \to 0$ as $t$, $w \to 0$.
Since $\delta - l \gamma \geq d$ and $l \gamma + d \geq d$,
it has a superattracting fixed point at the origin.
\end{prop}

Although $\tilde{f}$ is not skew product,
it is a perturbation of a monomial map near the origin.
Hence we can construct the B\"{o}ttcher coordinate for $\tilde{f}$
by similar arguments in Section 6 of this paper,
or one may refer to \cite[pp.498-499]{f}. 
This conjugacy induces the B\"{o}ttcher coordinate for $f$ on $U_r$.

\subsection{Newton polygons}

It is helpful to consider the Newton polygon of $\tilde{q}$.
Let 
\[
\tilde{d} = l_2^{-1} \gamma + d
\text{ and } 
\tilde{j} = l_2^{-1} i + j.
\]
Then $\tilde{q}(t,w) = t^{\gamma} w^{\tilde{d}} + \sum b_{ij} t^{i} w^{\tilde{j}}$,
and Lemma \ref{lem2 for main lem: case3} is translated into the following.

\begin{lem} \label{} 
It follows that
$\tilde{d} \leq \tilde{j}$ for any $(i,j)$ such that $b_{ij} \neq 0$.
\end{lem}

Therefore,
the Newton polygon of $\tilde{q}$ has just one vertex $(\gamma, \tilde{d})$:
$N(\tilde{q}) = D(\gamma, \tilde{d})$.

\begin{rem}
The linear transformation 
\[
A_2 
\begin{pmatrix} i \\ j \end{pmatrix} =
\begin{pmatrix} i \\ l_2^{-1} i + j \end{pmatrix} =
\begin{pmatrix} 1 & 0 \\ l_2^{-1} & 1 \end{pmatrix}
\begin{pmatrix} i \\ j \end{pmatrix}
\]
maps the basis $\{ (1,-l_2^{-1}), (0,1) \}$ to $ \{ (1,0), (0,1) \}$.
In other words, 
$A_2$ maps the line $L_{1}$ with slope $-l_2^{-1}$ and a vertical line 
to a horizontal line and the same vertical line.
\end{rem}

\section{Blow-ups, Newton polygons and Main lemma for Case 4}

We prove Lemma \ref{main lemma} for Case 4 in this section,
which completes the proof of the lemma.
Let $T_k \leq \delta \leq T_{k-1}$ for some $2 \leq k \leq s-1$,
\[
(\gamma, d) = (n_k, m_k), \
l_1 = \frac{n_k - n_{k-1}}{m_{k-1} - m_k} \text{ and } 
l_1 + l_2= \frac{n_{k+1} - n_k}{m_k - m_{k+1}}.
\] 
Note that $\delta > d$ and $\gamma > 0$ by the setting.
Let $f(z, w) = \left( z^{\delta}, z^{\gamma} w^{d} + \sum b_{ij} z^{i} w^{j} \right)$.
Against the previous two sections,
we first explain our results in terms of blow-ups and of Newton polygons
in Section 4.1,
and then prove Lemma \ref{main lemma} in Section 4.2.

\subsection{Blow-ups}

Assuming that $l_1$ and $l_2^{-1}$ are integer,
we blow-up $f$ to a nice superattracting germ
for which the B\"{o}ttcher coordinate exists on a neighborhood of the origin.

The strategy is to combine the blow-ups in Cases 2 and 3.
We first blow-up $f$ to $\tilde{f}_1$
by $\pi_1$ as in Case 2.
It then turns out that $\tilde{f}_1$ is a holomorphic skew product in Case 3.
We next blow-up $\tilde{f}_1$ to $\tilde{f}_2$
by $\pi_2$ as in Case 3.
The map $\tilde{f}_2$ is a perturbation of a monomial map near the origin,
and we obtain the B\"{o}ttcher coordinate.

\subsubsection{First blow-up}

We have the same inequalities as in Case 2. 

\begin{lem} \label{} 
It follows that
$l_1 \delta \leq \gamma + l_1 d \leq i + l_1 j$ for any $(i,j)$ such that $b_{ij} \neq 0$.
\end{lem}

\begin{proof}
These numbers $l_1 \delta$, $\gamma + l_1 d$ and $i + l_1 j$ are 
the $x$-intercepts of the lines with slope $-l_1^{-1}$
passing the points $(0, \delta)$, $(\gamma, d)$ and $(i,j)$.
\end{proof}

Let $\tilde{\gamma} = \gamma + l_1 d - l_1 \delta$
and $\tilde{i} = i + l_1 j - l_1 \delta$
as in Case 2.

\begin{lem} \label{inequalities for first blow-up} 
It follows that
$0 \leq \tilde{\gamma} \leq \tilde{i}$ for any $(i,j)$ such that $b_{ij} \neq 0$.
\end{lem}

More precisely,
$(\tilde{\gamma}, d)$ is minimum in the sense that
$\tilde{\gamma} \leq \tilde{i}$, and $d \leq j$ if $\tilde{\gamma} = \tilde{i}$.

Let $\pi_1 (z,c) = (z, z^{l_1} c)$
and $\tilde{f}_1 = \pi_1^{-1} \circ f \circ \pi_1$ as in Case 2.
Then 
\[
\tilde{f}_1 (z,c) 
= (\tilde{p}_1 (z), \tilde{q}_1 (z,c))
= \left( p(z), \ \dfrac{q(z,z^{l_1} c)}{p(z)^{l_1}} \right) 
\]
\[
= \left( z^{\delta}, z^{\gamma+l_1 d - l_1 \delta} c^d 
+ \sum b_{ij} z^{i + l_1 j - l_1 \delta} c^{j} \right)
= \left( z^{\delta}, z^{\tilde{\gamma}} c^d 
+ \sum b_{ij} z^{\tilde{i}} c^{j} \right).
\]

\begin{prop} \label{} 
If $l_1 \in \mathbb{N}$, 
then $\tilde{f}_1$ is well-defined, holomorphic and skew product 
on a neighborhood of the origin.  
More precisely,
\[
\tilde{f}_1 (z,c) =
\left( z^{\delta}, z^{\tilde{\gamma}} c^d + \sum b_{ij} z^{\tilde{i}} c^{j} \right),
\]
and it has a superattracting fixed point at the origin.
\end{prop}

Note that 
$(\tilde{\gamma}, d)$ is the vertex of the Newton polygon $N(\tilde{q}_1)$
whose $x$-coordinate is minimum,
and that
$N(\tilde{q}_1)$ has other vertices such as $(\tilde{n}_{k+1}, m_{k+1})$.
Hence the situation resembles that of Case 3.

We illustrate that $\tilde{f}_1$ is actually in Case 3. 
Recall that $L_{k}$ is the line passing 
the vertices $(\gamma, d)$ and $(n_{k+1}, m_{k+1})$,
and $T_{k}$ is the $y$-intercept of $L_{k}$.
The slope of $L_{k}$ is $-(l_1 + l_2)^{-1}$ 
and so $T_{k} = (l_1 + l_2)^{-1} \gamma + d$.
Let $\tilde{L}_{k}$ be the line passing 
the vertices $(\tilde{\gamma}, d)$ and $(\tilde{n}_{k+1}, m_{k+1})$,
and $\tilde{T}_{k}$ the $y$-intercept of $\tilde{L}_{k}$,
where $\tilde{n}_{k+1} = n_{k+1} + l_1 m_{k+1} - l_1 \delta$.
Then the slope of $\tilde{L}_{k}$ is $-l_2^{-1}$ 
and so $\tilde{T}_{k} = l_2^{-1} \tilde{\gamma} + d$.
The condition $T_k \leq \delta$ implies the following lemma and proposition.

\begin{lem} \label{} 
It follows that
$\tilde{T}_k \leq \delta$. 
More precisely,
$\tilde{T}_k < \delta$ if $T_k < \delta$,
and $\tilde{T}_k = \delta$ if $T_k = \delta$.
\end{lem}

\begin{proof}
Since $T_k = (l_1 + l_2)^{-1} \gamma + d \leq \delta$,
$\gamma + (l_1 + l_2)d \leq (l_1 + l_2) \delta$ 
and so $\gamma + l_1 d - l_1 \delta + l_2 d \leq l_2 \delta$.
Hence $\tilde{T}_k = l_2^{-1} \tilde{\gamma} + d = l_2^{-1} (\gamma + l_1 d - l_1 \delta) + d \leq \delta$.
\end{proof}

\begin{prop} \label{} 
If $l_1 \in \mathbb{N}$, 
then $\tilde{f}_1$ is a holomorphic skew product in Case 3. 
\end{prop}

\subsubsection{Second blow-up}

We have the same inequalities as in Case 3 for $\tilde{\gamma}$ and $\tilde{i}$,
instead for $\gamma$ and $i$.

\begin{lem} \label{} 
It follows that
$l_2^{-1} \tilde{\gamma} + d \leq l_2^{-1} \tilde{i} + j$ 
and $l_2^{-1} \tilde{\gamma} + d \leq \delta$ 
for any $(i,j)$ such that $b_{ij} \neq 0$.
\end{lem}

\begin{proof}
The numbers $l_2^{-1} \tilde{\gamma} + d$ and $l_2^{-1} \tilde{i} + j$ are 
the $y$-intercepts of the lines with slope $-l_2^{-1}$
passing the points $(\tilde{\gamma}, d)$ and $(\tilde{i},j)$.
In particular,
$l_2^{-1} \tilde{\gamma} + d = \tilde{T}_k \leq \delta$.
\end{proof}

Let $\tilde{d} = l_2^{-1} \tilde{\gamma} + d$
and $\tilde{j} = l_2^{-1} \tilde{i} + j$
as in Case 3.

\begin{lem} \label{inequalities for second blow-up} 
It follows that
$\tilde{d} \leq \tilde{j}$ and $\tilde{d} \leq \delta$ for any $(i,j)$ such that $b_{ij} \neq 0$.
\end{lem}

In particular,
the minimality of $(\tilde{\gamma}, \tilde{d})$ follows from 
Lemmas \ref{inequalities for first blow-up} and \ref{inequalities for second blow-up}.

\begin{cor}\label{cor for second blow-up}
It follows that
$0 \leq \tilde{\gamma} \leq \tilde{i}$ and $d \leq \tilde{d} \leq \tilde{j}$ for any $(i,j)$ such that $b_{ij} \neq 0$.
\end{cor}

Let $\pi_2 (t,c) = (t c^{l_2^{-1}}, c)$ and $\tilde{f}_2 = \pi_2^{-1} \circ \tilde{f}_1 \circ \pi_2$ as in Case 3.
Then 
\[
\tilde{q}_2 (t,c) = \tilde{q}_1 (tc^{l_2^{-1}},c)
= (tc^{l_2^{-1}})^{\tilde{\gamma}} c^d + \sum b_{ij} (tc^{l_2^{-1}})^{\tilde{i}} c^{j} 
= t^{\tilde{\gamma}} c^{\tilde{d}} + \sum b_{ij} t^{\tilde{i}} c^{\tilde{j}}
\]
\[
= t^{\tilde{\gamma}} c^{\tilde{d}} \left\{ 1 + \sum b_{ij} t^{\tilde{i} - \tilde{\gamma}} c^{\tilde{j} - \tilde{d}} \right\}
= t^{\tilde{\gamma}} c^{\tilde{d}} \left\{ 1 + \eta_2 (t,c) \right\}
\sim t^{\tilde{\gamma}} c^{\tilde{d}}
\text{ and so}
\]
\[
\tilde{f}_2 (t,c) 
= (\tilde{p}_2 (t,c), \tilde{q}_2 (t,c))
= \left( \dfrac{\tilde{p}_1 (tc^{l_2^{-1}})}{\tilde{q}_1 (tc^{l_2^{-1}},c)^{l_2^{-1}}}, \ \tilde{q}_1 (tc^{l_2^{-1}},c) \right)
\]
\[
= \left( \dfrac{ t^{\delta - l_2^{-1} \tilde{\gamma}} c^{l_2^{-1} (\delta - \tilde{d})} }{ \{ 1 + \eta_2 (t,c) \}^{l_2^{-1}} },
\ t^{\tilde{\gamma}} c^{\tilde{d}} \{ 1 + \eta_2 (t,c) \} \right)
\sim \left( t^{\delta - l_2^{-1} \tilde{\gamma}} c^{l_2^{-1} (\delta - \tilde{d})}, \ t^{\tilde{\gamma}} c^{\tilde{d}} \right).
\]

\begin{prop} \label{} 
If $l_1, l_2^{-1} \in \mathbb{N}$, 
then $\tilde{f}_2$ is well-defined, holomorphic and rigid
on a neighborhood $\{ |t| < r, |c| < r \}$ of the origin. 
More precisely,
\[
\tilde{f}_2 (t,c) =
\left( t^{\delta - l_2^{-1} \tilde{\gamma}} c^{l_2^{-1} (\delta - \tilde{d})} \{ 1 + \zeta_2 (t,c) \}, 
\ t^{\tilde{\gamma}} c^{\tilde{d}} \{ 1 + \eta_2 (t,c) \} \right),
\]
where $\zeta_2$, $\eta_2 \to 0$ as $t$, $c \to 0$.
Since $\delta - l_2^{-1} \tilde{\gamma} \geq \tilde{d} \geq d$,
it has a superattracting fixed point at the origin.
\end{prop}

Therefore,
we can construct the B\"{o}ttcher coordinate for $\tilde{f}_2$
on $\{ |t| < r, |c| < r \}$,
which induces that for $\tilde{f}_1$ 
on $\{ |z| < r |c|^{l_2^{-1}}, |c| < r \}$ and 
that for $f$ 
on $U_r$.

\subsubsection{Newton polygons}

Although the Newton polygon of $\tilde{q}_1$ has at least two vertices,
the Newton polygon of $\tilde{q}_2$ has just one vertex $(\tilde{\gamma}, \tilde{d})$:
$N(\tilde{q}_2) = D(\tilde{\gamma}, \tilde{d})$.

\begin{rem}
The affine transformation
\[
A 
\begin{pmatrix} i \\ j \end{pmatrix} =
\begin{pmatrix} 1 & 0 \\ l_2^{-1} & 1 \end{pmatrix}
\left\{  
\begin{pmatrix} 1 & l_1 \\ 0 & 1 \end{pmatrix}
\begin{pmatrix} i \\ j \end{pmatrix} - 
\begin{pmatrix} l_1 \delta \\ 0 \end{pmatrix}
\right\}
\]
is the composition of the two affine transformations
\[
A_1
\begin{pmatrix} i \\ j \end{pmatrix} =
\begin{pmatrix} i + l_1 j - l_1 \delta \\ j \end{pmatrix} 
\text{ and }
A_2
\begin{pmatrix} i \\ j \end{pmatrix} =
\begin{pmatrix} i \\ l_2^{-1} i + j \end{pmatrix}. 
\]
The transformation $A_1$ maps 
the basis $\{ (1,-(l_1 +l_2)^{-1}), (-l_1,1) \}$ to $\{ (1,-l_2^{-1}), (0,1) \}$.
In other words, 
it maps the line $L_{k}$ with slope $-(l_1 +l_2)^{-1}$ and the line $L_{k-1}$ with slope $-l_1^{-1}$,
which intersect with $(\gamma, d)$,
to the line $\tilde{L}_{k}$ with slope $-l_2^{-1}$ and the vertical line,
which intersect with $(\tilde{\gamma}, d)$.
The transformation $A_2$ maps 
the basis $\{ (1,-l_2^{-1}), (0,1) \}$ to $ \{ (1,0), (0,1) \}$.
In other words, 
it maps the line $\tilde{L}_{k}$ and the vertical line 
to the horizontal line and the vertical line,
which intersect with $(\tilde{\gamma}, \tilde{d})$.
Therefore,
$A$ maps the lines $L_{k}$ and $L_{k-1}$
to the horizontal and vertical lines.
\end{rem}

\subsection{Proof of the main lemma}

The idea of the blow-ups in the previous subsection
provides a proof of Lemma \ref{main lemma}.
Because we take the absolute value in the proof,
we do not need to care whether $\tilde{f}_1$ and $\tilde{f}_2$ are well-defined.

\begin{proof}[Proof of Lemma \ref{main lemma} for Case 4] 
We first define $\eta (z,w) = (q(z,w) - z^{\gamma} w^d)/ z^{\gamma} w^d$
and show the former statement.
Let $|w| = |z^{l_1}c|$ and $|z| = |tc^{l_2^{-1}}|$.
Then 
\[
U_r = \{ |z|^{l_1 + l_2} < r^{l_2} |w|, |w| < r|z|^{l_1} \}
  = \{ 0 < |z| < r |c|^{l_2^{-1}}, 0 < |c| < r \}
\]
\[
  = \{ 0 < |t| < r, 0 < |c| < r \}
  \subset \{ |t| < r, |c| < r \},
\]
\[
\left| \frac{z^{i} w^{j}}{z^{\gamma} w^{d}} \right| 
= \left| \frac{z^{i} (z^{l_1}c)^{j}}{z^{\gamma} (z^{l_1}c)^{d}} \right| 
= \left| \frac{z^{i +l_1 j} c^{j}}{z^{\gamma + l_1 d}c^{d}} \right|
= \left| \frac{z^{\tilde{i}} c^{j}}{z^{\tilde{\gamma}}c^{d}} \right|
\]
\[
= \left| \frac{(tc^{l_2^{-1}})^{\tilde{i}} c^{j}}{(tc^{l_2^{-1}})^{\tilde{\gamma}} c^{d}} \right|
= \left| \frac{t^{\tilde{i}} c^{l_2^{-1} \tilde{i} + j}}{t^{\tilde{\gamma}} c^{l_2^{-1} \tilde{\gamma} + d}} \right|
= \left| \frac{t^{\tilde{i}} c^{\tilde{j}}}{t^{\tilde{\gamma}} c^{\tilde{d}}} \right|
\text{ and so}
\]
\[
|\eta|
\leq \sum |b_{ij}| |t|^{\tilde{i} - \tilde{\gamma}} |c|^{\tilde{j} - \tilde{d}}.
\]
The inequalities $\tilde{i} \geq \tilde{\gamma}$ 
and $\tilde{j} \geq \tilde{d}$ in Corollary \ref{cor for second blow-up} ensure that
the left-hand side is a power series in $|t|$ and $|c|$, 
and so converges on $\{ |t| < r, |c| < r \}$.
Moreover, 
at least one of the inequalities $\tilde{i} - \tilde{\gamma} > 0$ 
and $\tilde{j} - \tilde{d} > 0$ holds.
Therefore, 
for any small $\varepsilon$
there is $r$ such that
$|\eta| < \varepsilon$ on $U_r$. 

We next show the invariance of $U_r$.
Note that $|z| = |tc^{l_2^{-1}}|$ and $|w| = |t^{l_1} c^{1 + l_1 l_2^{-1}}|$
and that, formally,
\[
\tilde{f}_2 (t,c) =
\left( \dfrac{\tilde{p}_1 (tc^{l_2^{-1}})}{\tilde{q}_1 (tc^{l_2^{-1}},c)^{l_2^{-1}}}, \ \tilde{q}_1 (tc^{l_2^{-1}},c) \right)
\]
\[
= \left( \dfrac{p(tc^{l_2^{-1}})^{1+l_1 l_2^{-1}}}{q(tc^{l_2^{-1}}, t^{l_1} c^{1+l_1 l_2^{-1}})^{l_2^{-1}}}, 
\dfrac{q(tc^{l_2^{-1}}, t^{l_1} c^{1+l_1 l_2^{-1}})}{p(tc^{l_2^{-1}})^{l_1}} \right).
\]
Because $\delta \geq \tilde{d} = l_2^{-1} \tilde{\gamma} + d$,
\[
\left| \frac{p(z)^{1 + l_1 l_2^{-1}}}{q(z,w)^{l_2^{-1}}} \right| 
\sim \left| \frac{(z^{\delta})^{1 + l_1 l_2^{-1}} }{(z^{\gamma} w^{d})^{l_2^{-1}}} \right| 
= \left| \frac{ \{ (tc^{l_2^{-1}})^{\delta} \}^{1 + l_1 l_2^{-1}} }{\{ (tc^{l_2^{-1}})^{\gamma} (t^{l_1} c^{1 + l_1 l_2^{-1}})^{d} \}^{l_2^{-1}}} \right| 
\]
\[
= |t|^{\delta - l_2^{-1} \tilde{\gamma}} |c|^{l_2^{-1} (\delta - \tilde{d})}
\leq |t|^{\delta - l_2^{-1} \tilde{\gamma}} \leq |t|^d < r^d,
\]
and because $\tilde{\gamma} \geq 0$ and $\tilde{d} \geq d$,
\[
\left| \frac{q(z,w)}{p(z)^{l_1}} \right| 
\sim \left| \frac{z^{\gamma} w^{d}}{(z^{\delta})^{l_1}} \right| 
= \left| \frac{(tc^{l_2^{-1}})^{\gamma} (t^{l_1} c^{1 + l_1 l_2^{-1}})^{d}}{ \{ (tc^{l_2^{-1}})^{\delta} \}^{l_1}} \right| 
\]
\[
= |t|^{\tilde{\gamma}} |c|^{\tilde{d}}
\leq |c|^{\tilde{d}} \leq |c|^d < r^d
\]
on $U_r$ as $r \to 0$.
\end{proof}

\section{Intervals of weights and Branched coverings}


The rational numbers $l_1$ and $l_2$ are called weights 
in the previous papers \cite{u3} and \cite{u2}, respectively. 
In this section 
we introduce intervals of weights for each of which Lemma \ref{main lemma} holds.
For Cases 2 and 3, 
the intervals are used to state the results in the previous papers,
instead of the Newton polygon.

Moreover,
we associate weights in the intervals to branched coverings of $f$.
These coverings are a generalization of the blow-ups of $f$ in the previous sections,
and it might be well-defined even if the weight is rational.  
We deal with Cases 2, 3 and 4 in Sections 5.1, 5.2 and 5.3, respectively.
For Case 2,
the branched covering is well-defined for any rational number in the interval;
see Proposition \ref{branched coverings: case2}.  
On the other hand,
for Cases 3 and 4,
the case when the branched covering is well-defined seems to be limited;
see Corollaries \ref{branched coverings: case3} and \ref{branched coverings: case4}, respectively. 

\subsection{Intervals and coverings for Case 2}

In the proof of Lemma \ref{main lemma} for Case 2,
the inequalities $l_1 \delta \leq \gamma + l_1 d \leq i + l_1 j$ played a central role.
We define the interval $\mathcal{I}_f$ as 
\[
\mathcal{I}_f = 
\left\{ \ l > 0 \ | 
\begin{array}{lcr}
l \delta \leq \gamma + ld \leq i + l j
\text{ for any $i$ and $j$ s.t. } b_{ij} \neq 0
\end{array}
\right\}.
\]
It follows that $\min \mathcal{I}_f = l_1$.
In fact, 
if $\delta > d$, then $\gamma > 0$ and
\[
\mathcal{I}_f 
=
\left[
\max \left\{ \dfrac{\gamma - i}{j - d} \right\},
\dfrac{\gamma}{\delta - d}
\right]
=
\left[
\frac{\gamma - n_{s-1}}{m_{s-1} - d},
\dfrac{\gamma}{\delta - d}
\right]
=
\left[
l_1,
\dfrac{\gamma}{\delta - d}
\right],
\]
which is mapped to $[\delta, T_{s-1}]$
by the transformation $l \to l^{-1} \gamma + d$.
If $\delta \leq d$, then
the inequality $l \delta \leq \gamma + ld$ is trivial and so
$\mathcal{I}_f =[ l_1, \infty )$.

Let $U^l = \{ |z| < r, |w| < r |z|^{l} \}$.

\begin{prop} \label{} 
Lemma \ref{main lemma} in Case 2 holds on $U^l$ for any $l$ in $\mathcal{I}_f$.
\end{prop}

\begin{rem} \label{} 
It follows 
that $U^{l_1}$ is the largest region among $U^l$ for any $l$ in $\mathcal{I}_f$, and
that $\mathcal{I}_f \neq \emptyset$ if and only if $\delta \leq T_{s-1}$.
\end{rem}


Let $\pi_1 (\mathsf{z},c) = (\mathsf{z}^r, \mathsf{z}^s c)$
and $\tilde{f} = \pi_1^{-1} \circ f \circ \pi_1$,
where $s/r = l \in \mathcal{I}_f \cap \mathbb{Q}$.
Then
$\pi_1$ is formally the composition of $(\mathsf{z},c) \to (\mathsf{z}^r,c)$ and $(z,c) \to (z,z^{s/r} c)$,
and $\tilde{f}$ is well-defined:
\[
\tilde{f} (\mathsf{z},c) 
= (\tilde{p} (\mathsf{z}), \tilde{q} (\mathsf{z},c)) 
= \left( p(\mathsf{z}^r)^{1/r}, \ \dfrac{q(\mathsf{z}^r,\mathsf{z}^s c)}{p(\mathsf{z})^{s/r}} \right), 
\]
\[
\tilde{p} (\mathsf{z}) = p(\mathsf{z}^r)^{1/r} = (\mathsf{z}^{r \delta})^{1/r} = \mathsf{z}^{\delta}
\text{ and}
\]
\[
\tilde{q} (\mathsf{z},c) 
= \mathsf{z}^{r \gamma + sd - s \delta} c^d
\left\{ 1 + \sum b_{ij} \mathsf{z}^{(ri + sj) - (r \gamma + sd)} c^{j - d} \right\}.
\]
Moreover,
$\tilde{f}$ is holomorphic 
since $ri + sj \geq r \gamma + sd \geq s \delta$.

\begin{prop} \label{branched coverings: case2} 
For any rational number $s/r$ in $\mathcal{I}_f$, 
the lift $\tilde{f}$ is well-defined, holomorphic, skew product and rigid 
on the preimage of the domain of $f$ by $\pi_1$. 
More precisely,
\[
\tilde{f} (\mathsf{z},c) = \left( \mathsf{z}^{\delta}, \ 
\mathsf{z}^{r \gamma + sd - s \delta} c^d \left\{ 1 + \eta (\mathsf{z},c) \right\}  \right),
\]
where $\eta \to 0$ as $\mathsf{z}$, $c \to 0$, and 
it has a superattracting fixed point at the origin.
\end{prop}

\subsection{Intervals and coverings for Case 3}

In the proof of Lemma \ref{main lemma} for Case 3,
the inequalities $\gamma + l_2 d \leq i + l_2 j$ and
$\gamma + l_2 d \leq l_2 \delta$ played a central role.
We define the interval $\mathcal{I}_f$ as
\[
\mathcal{I}_f = 
\left\{ \ l > 0 \ | 
\begin{array}{lcr}
\gamma + ld \leq i + lj
\text{ and }
\gamma + ld \leq l \delta  
\text{ for any $i$ and $j$ s.t. } b_{ij} \neq 0
\end{array}
\right\}. 
\]
It follows that $\max \mathcal{I}_f = l_2$.
In fact,
if $\gamma > 0$, then $\delta > d$ and 
\[
\mathcal{I}_f 
= 
\left[
\dfrac{\gamma}{\delta - d},
\min \left\{ \dfrac{i - \gamma}{d - j}  \right\}
\right]
=
\left[
\dfrac{\gamma}{\delta - d},
\frac{n_2 - \gamma}{d - m_2}
\right]
=
\left[
\dfrac{\gamma}{\delta - d},
l_2
\right],
\]
which is mapped to $[T_1, \delta]$
by the transformation $l \to l^{-1} \gamma + d$.
If $\gamma = 0$, then
the inequality $\gamma + ld \leq l \delta$ is trivial since $d \leq \delta$,
and so $\mathcal{I}_f = (0, l_2]$.

Let $U^l = \{ |z|^{l} < r^{l} |w|, |w| < r \}$.

\begin{prop} \label{} 
Lemma \ref{main lemma} in Case 3 holds on $U^l$ for any $l$ in $\mathcal{I}_f$.
\end{prop}

\begin{rem} \label{} 
It follows
that $U^{l_2}$ is the largest region among $U^l$ for any $l$ in $\mathcal{I}_f$, and
that $\mathcal{I}_f \neq \emptyset$ if and only if $T_1 \leq \delta$.
\end{rem}


Let $\pi_2 (t, \mathsf{w}) = (t \mathsf{w}^r, \mathsf{w}^s)$ and $\tilde{f} = \pi_2^{-1} \circ f \circ \pi_2$,
where $s/r = l \in \mathcal{I}_f \cap \mathbb{Q}$.
Then, formally,
$\pi_2$ is the composition of $(t,\mathsf{w}) \to (t,\mathsf{w}^s)$ and $(t,w) \to (tw^{r/s},w)$, and
\[
\tilde{f} (t, \mathsf{w}) =
\left( \dfrac{p(t \mathsf{w}^r)}{q(t \mathsf{w}^r, \mathsf{w}^s)^{r/s}}, \ q(t \mathsf{w}^r, \mathsf{w}^s)^{1/s} \right).
\]
Since $q(z,w) \sim z^{\gamma} w^d$ on $U^l$,
it follows formally that 
\[
q(t \mathsf{w}^r, \mathsf{w}^s)^{1/s} 
\sim \{ (t \mathsf{w}^r)^{\gamma} (\mathsf{w}^s)^d \}^{1/s} = (t \mathsf{w}^r)^{\gamma /s} \mathsf{w}^d
\text{ on } \pi_2^{-1} (U^l).
\]
Hence, 
if $\gamma /s$ is integer,
then $\tilde{f}$ is well-defined on $\pi_2^{-1} (U^l) = \{ |t| < r, 0 < |w| < r \}$.

\begin{prop} \label{} 
If $s/r \in \mathcal{I}_f$ and $\gamma /s \in \mathbb{N}$, 
then $\tilde{f}$ is well-defined, holomorphic and rigid
on a neighborhood of the origin.
More precisely,
\[
\tilde{f} (t, \mathsf{w}) = 
\left( t^{\delta - \frac{r}{s} \gamma} \mathsf{w}^{\frac{r}{s} \{ s \delta - (r \gamma + sd) \}} \{ 1 + \zeta (t, \mathsf{w}) \},
\ t^{\frac{\gamma}{s}} \mathsf{w}^{\frac{r}{s} \gamma + d} \{ 1 + \eta (t, \mathsf{w}) \} \right),
\]
where $\zeta$, $\eta \to 0$ as $t$, $\mathsf{w} \to 0$, 
and it has a superattracting fixed point at the origin.
\end{prop}

\begin{cor} \label{branched coverings: case3}
If $\gamma = 0$, 
then $\tilde{f}$ is well-defined for any $s/r$ in $\mathcal{I}_f$.
If $\gamma > 0$, 
then $\tilde{f}$ is well-defined at least for $\gamma /(\delta - d)$.
\end{cor}

\subsection{Intervals and coverings for Case 4}

We define the interval $\mathcal{I}_f^1$ as
\[
\mathcal{I}_f^1 = 
\left\{ \ l_{(1)} > 0 \ \Bigg| 
\begin{array}{lcr}
\gamma + l_{(1)} d \leq n_{j} + l_{(1)} m_{j} \text{ for } j \leq k-1 \\ 
\gamma + l_{(1)} d  <   n_{j} + l_{(1)} m_{j} \text{ for } j \geq k+1 \\ 
l_{(1)} \delta \leq \gamma + l_{(1)} d 
\end{array}
\right\},
\]
the interval $\mathcal{I}_f^2$ associated with $l_{(1)}$ in $\mathcal{I}_f^1$ as
\[
\mathcal{I}_f^2  = \mathcal{I}_f^2 (l_{(1)}) = 
\left\{ \ l_{(2)} > 0 \ \Big| 
\begin{array}{lcr}
\tilde{\gamma} + l_{(2)} d \leq \tilde{i} + l_{(2)} j
\text{ and }
\tilde{\gamma} + l_{(2)} d \leq l_{(2)} \delta  \\
\text{ for any $i$ and $j$ s.t. } b_{ij} \neq 0
\end{array}
\right\},
\]
where 
$\tilde{\gamma} = \gamma + l_{(1)} d - l_{(1)} \delta$
and $\tilde{i} = i + l_{(1)} j - l_{(1)} \delta$,
and the rectangle $\mathcal{I}_f$ as
\[
\mathcal{I}_f =
\{ (l_{(1)}, l_{(1)} + l_{(2)}) \ | \ l_{(1)} \in \mathcal{I}_f^1, l_{(2)} \in \mathcal{I}_f^2 \}.
\]

Let us calculate the intervals and rectangle more practically.
Let 
\[
\alpha_0 = \frac{\gamma}{\delta - d}.
\]
Then $\alpha_0 > 0$ since $\delta > d$ and $\gamma > 0$ by the setting.
Since $n_j < \gamma$ and $m_j > d$ for any $j \leq k-1$,
and $n_j > \gamma$ and $m_j < d$ for any $j \geq k+1$,  
\[
\mathcal{I}_f^1 = 
\left[
\max_{j \leq k-1}
\left\{ 
\dfrac{\gamma - n_j}{m_j - d} 
\right\},
\min_{j \geq k+1}
\left\{ 
\dfrac{n_j - \gamma}{d - m_j} 
\right\}
\right)
\cap
\left(
0,
\dfrac{\gamma}{\delta - d}
\right]
\]
\[
= 
\left[
\dfrac{\gamma - n_{k-1}}{m_{k-1} - d},
\dfrac{n_{k+1} - \gamma}{d - m_{k+1}} 
\right)
\cap
\left(
0,
\dfrac{\gamma}{\delta - d}
\right]
= [ l_1, l_1 + l_2 ) \cap ( 0, \alpha_0 ].
\]
In particular,
$\min \mathcal{I}_f^1 = l_1$ and, as a remark,
\[
\mathcal{I}_f^1 = 
\left\{ \ l_{(1)} > 0 \ \Bigg| 
\begin{array}{lcr}
\gamma + l_{(1)} d \leq n_{k-1} + l_{(1)} m_{k-1} \\
\gamma + l_{(1)} d < n_{k+1} + l_{(1)} m_{k+1} \\
l_{(1)} \delta \leq \gamma + l_{(1)} d 
\end{array}
\right\}.
\]
On the other hand,
\[
\mathcal{I}_f^2 
=
\left[
\dfrac{\tilde{\gamma}}{\delta - d},
\dfrac{\tilde{n}_{k+1} - \tilde{\gamma}}{d - m_{k+1}}  
\right]
=
\left[
\dfrac{\gamma}{\delta - d} - l_{(1)},
\dfrac{n_{k+1} - \gamma}{d - m_{k+1}} - l_{(1)}  
\right]
\]
\[
= [ \alpha_0 - l_{(1)}, l_1 + l_2 - l_{(1)} ].
\]
If $T_k < \delta = T_{k-1}$, 
then it follows from the inequality $l_1 = \alpha_0 < l_1 + l_2$ that
\[
\mathcal{I}_f^1 = \{ l_1 \}, \ 
\mathcal{I}_f^2 = [ \alpha_0 - l_{(1)}, l_1 + l_2 - l_{(1)} ]
\text{ and so }
\mathcal{I}_f 
= \{ l_1 \} \times [ l_1, l_1 + l_2 ].
\]
If $T_k < \delta < T_{k-1}$, 
then it follows from the inequality $l_1 < \alpha_0 < l_1 + l_2$ that
\[
\mathcal{I}_f^1 = [ l_1, \alpha_0 ], \ 
\mathcal{I}_f^2 = [ \alpha_0 - l_{(1)}, l_1 + l_2 - l_{(1)} ]
\text{ and so }
\mathcal{I}_f 
= [ l_1, \alpha_0 ] \times [ \alpha_0, l_1 + l_2 ].
\]
If $T_{k} = \delta < T_{k-1}$, then it follows from the inequality $l_1 < \alpha_0 = l_1 + l_2$ that
\[
\mathcal{I}_f^1 = [ l_1, l_1 + l_2 ), \ 
\mathcal{I}_f^2 = \{ l_1 + l_2 - l_{(1)} \}
\text{ and so }
\mathcal{I}_f 
= [ l_1, l_1 + l_2 ) \times \{ l_1 + l_2 \}.
\]
In particular, 
$\min \mathcal{I}_f^1 = l_1$ and
$\max \{ l_{(1)} + l_{(2)} \ | \ l_{(1)} \in \mathcal{I}_f^1, l_{(2)} \in \mathcal{I}_f^2 \} = l_1 + l_2$.

Let $U^{l_{(1)},l_{(2)}} = \{ |z|^{l_{(1)} + l_{(2)}} < r^{l_{(2)}} |w|, |w| < r|z|^{l_{(1)}} \}$.

\begin{prop} \label{} 
Lemma \ref{main lemma} in Case 4 holds on $U^{l_{(1)},l_{(2)}}$
for any $l_{(1)}$ in $\mathcal{I}_f^1$ and $l_{(2)}$ in $\mathcal{I}_f^2$.
\end{prop}

\begin{rem} \label{} 
It follows that 
$U^{l_1, l_2}$ is the largest region among $U^{l_{(1)},l_{(2)}}$ 
for any $l_{(1)}$ in $\mathcal{I}_f^1$ and $l_{(2)}$ in $\mathcal{I}_f^2$,
and that
$\mathcal{I}_f^1 \neq \emptyset$ and $\mathcal{I}_f^2 \neq \emptyset$
if and only if $T_{k} \leq \delta \leq T_{k-1}$.
More precisely,
$\mathcal{I}_f^1 = \emptyset$ if $T_{k-1} < \delta$, and
$\mathcal{I}_f^2 = \emptyset$ if $\delta < T_k$.
\end{rem}


Let $\pi_1 (\mathsf{z},c) = (\mathsf{z}^{r_1}, \mathsf{z}^{s_1} c)$
and $\tilde{f}_1 = \pi_1^{-1} \circ f \circ \pi_1$.
Then 
\[
\tilde{f}_1 (\mathsf{z},c) 
= \left( \mathsf{z}^{\delta}, \mathsf{z}^{r_1 \gamma + s_1 d - s_1 \delta} c^d 
+ \sum b_{ij} \mathsf{z}^{r_1 i + s_1 j - s_1 \delta} c^{j} \right)
\]
\[
= \left( \mathsf{z}^{\delta}, \mathsf{z}^{\tilde{\gamma}} c^d 
+ \sum b_{ij} \mathsf{z}^{\tilde{i}} c^{j} \right),
\]
where $\tilde{\gamma} = r_1 \gamma + s_1 d - s_1 \delta$
and $\tilde{i} = r_1 i + s_1 j - s_1 \delta$.

\begin{prop} \label{} 
For any rational number $s_1/r_1$ in $\mathcal{I}_f^1$, 
the lift $\tilde{f}_1$ is well-defined, holomorphic, skew product and rigid
on the preimage of the domain of $f$ by $\pi_1$.
More precisely,
\[
\tilde{f}_1 (\mathsf{z},c) = \left( \mathsf{z}^{\delta}, \ 
\mathsf{z}^{r_1 \gamma + s_1 d - s_1 \delta} c^d \left\{ 1 + \eta (\mathsf{z},c) \right\}  \right),
\]
where $\eta \to 0$ as $\mathsf{z}$, $c \to 0$, and 
it has a superattracting fixed point at the origin.
\end{prop}

\begin{rem}
If we define the interval $\mathcal{I}_f^1$ as
\[
\left\{ \ l_{(1)} > 0 \ | 
\begin{array}{lcr}
l_{(1)} \delta \leq \gamma + l_{(1)} d \leq i + l_{(1)} j
\text{ for any $i$ and $j$ s.t. } b_{ij} \neq 0
\end{array}
\right\},
\]
then we could have the equality $\tilde{\gamma} = \tilde{n}_{k+1}$
and the proposition above fails.
\end{rem}


Let $\pi_2 (t, \mathsf{c}) = (t \mathsf{c}^{r_2}, \mathsf{c}^{s_2})$ 
and $\tilde{f}_2 = \pi_2^{-1} \circ \tilde{f}_1 \circ \pi_2$.
Then, formally,
\[
\tilde{f}_2 (t, \mathsf{c}) =
\left( \dfrac{\tilde{p}_1 (t \mathsf{c}^{r_2})}{\tilde{q}_1 (t \mathsf{c}^{r_2}, \mathsf{c}^{s_2})^{{r_2}/{s_2}}}, \ 
\tilde{q}_1(t \mathsf{c}^{r_2}, \mathsf{c}^{s_2})^{1/{s_2}} \right).
\]

\begin{lem} \label{} 
If $s_1/r_1 \in \mathcal{I}_f^1$, $s_2/r_2 \in \mathcal{I}_f^2$
and $\tilde{\gamma} /{s_2} \in \mathbb{N}$,  
then $\tilde{f}_2$ is well-defined, holomorphic and rigid 
on a neighborhood of the origin. 
More precisely,
\[
\tilde{f}_2 (t, \mathsf{c}) = 
\left( t^{\delta - \frac{r_2}{s_2} \tilde{\gamma}} 
 \mathsf{c}^{\frac{r_2}{s_2} \{ s_2 \delta - (r_2 \tilde{\gamma} + s_2 d) \}} 
 \{ 1 + \zeta (t, \mathsf{c}) \},
\ t^{\frac{\tilde{\gamma}}{s_2}} 
 \mathsf{c}^{\frac{r_2}{s_2} \tilde{\gamma} + d} 
 \{ 1 + \eta (t, \mathsf{c}) \} \right),
\]
where $\zeta$, $\eta \to 0$ as $t$, $\mathsf{c} \to 0$,
and it has a superattracting fixed point at the origin.
\end{lem}

\begin{cor} \label{branched coverings: case4}
Let $T_k < \delta \leq T_{k-1}$
and $s_1/r_1 = \gamma/(\delta - d)$.
Then $\tilde{f}_2$ is well-defined
for any $s_2/r_2$ in $\mathcal{I}_f^2$.
\end{cor}

\begin{proof}
It follows from the condition $T_k < \delta$
that $l_1 + l_2 < \alpha_0$
and so $\mathcal{I}_f^1 = [l_1, \alpha_0]$.
In particular,
$\alpha_0 = \gamma/(\delta - d) \in \mathcal{I}_f^1$. 
Let $s_1/r_1 = \gamma/(\delta - d)$.
Then $\tilde{\gamma} = 0$ and so 
we obtain the condition $\tilde{\gamma} /s_2 = 0 \in \mathbb{N}$
in the previous lemma.
\end{proof}

\begin{rem}
Even if $\tilde{f}_2$ is well-defined,
the projection under $\pi_1 \circ \pi_2$
of a neighborhood of the origin
is usually smaller than the open set $U_r$ 
in the following sense:
\[
\pi_1 (\pi_2 (\{ 0 < |t| < r, 0 < |\mathsf{c}| < r \}))
= \pi_1 (\{ 0 < r^{- l_2} |\mathsf{z}|^{l_2} < |c| < r^{s_2} \})
\]
\[
= \{ r^{-l_2} |z|^{l_1 + l_2/r_1} < |w| < r^{s_2} |z|^{l_1} \}
\subset U_r = \{ r^{-l_2} |z|^{l_1 + l_2} < |w| < r|z|^{l_1} \}
\]
and, in particular,
$l_1 + l_2/r_1 < l_1 + l_2$ if $r_1 \geq 2$.
\end{rem}

\section{Proof of Main Theorem}

Theorem \ref{main theorem} follows from Lemma \ref{main lemma}
by the same arguments as in \cite{u2},
which are denoted again for the completeness. 
Let 
\[
f(z,w) =
(z^{\delta} + O(z^{\delta + 1}), z^{\gamma} w^{d} + \sum b_{ij} z^{i} w^{j})
\]
and $d \geq 2$.
We prove that
the composition $\phi_n = f_0^{-n} \circ f^n$ is well-defined on $U_r$ in Section 6.1,
converges uniformly to $\phi$ on $U_r$ in Section 6.2,
and the limit $\phi$ is injective on $U_r$ in Section 6.3.
Although the injectivity of the lift $\Phi$ of $\phi$ was proved in \cite{u2, u3},
we prove the injectivity of the lift $F$ of $f$ in this paper,
and obtain a larger region that ensures the injectivity of $\phi$ 
as stated in \cite[Remark 4.3]{u2}.

\subsection{Well-definedness of $\phi_n$}

Thanks to Lemma \ref{main lemma},
we may write 
\[
p(z) = z^{\delta} (1 + \zeta (z)) \text{ and } 
q(z,w) = z^{\gamma} w^d (1 + \eta (z,w)),
\]
where $\zeta$ and $\eta$ are holomorphic on $U_r$ and
converge to $0$ on $U_r$ as $r \to 0$. 
Then the first and second components of $f^n$ are written as
\[
z^{\delta^n} \prod_{j = 1}^{n} (1 + \zeta (p^{j - 1} (z)))^{\delta^{n-j}}
\text{ and}
\]
\[
z^{\gamma_n} w^{d^n} 
\prod_{j = 1}^{n - 1} (1 + \zeta (p^{j - 1} (z)))^{\gamma_{n-j}} 
\prod_{j = 1}^{n} (1 + \eta (f^{j - 1} (z,w)))^{d^{n-j}}, 
\]
where $\gamma_n = \sum_{j=1}^{n} \delta^{n-j} d^{j-1} \gamma$.
Since
$f_0^{-n} (z,w) = (z^{1 / \delta^n}, z^{- \gamma_n / \delta^n d^n} w^{1 / d^n})$,
we can define $\phi_n$ as follows:
\[ 
\phi_n (z,w)
= \left( z \cdot \prod_{j = 1}^{n} \sqrt[\delta^j]{1 + \zeta (p^{j - 1} (z))},
w \cdot \prod_{j = 1}^{n} \frac{\sqrt[d^j]{1 + \eta (f^{j - 1} (z,w))}}
{\sqrt[(\delta d)^j]{\{ 1 + \zeta (p^{j - 1} (z)) \}^{\gamma_j}}} \right),
\]
which is well-defined and so holomorphic on $U_r$.

\subsection{Uniform convergence of $\phi_n$}

In order to prove the uniform convergence of $\phi_n$,
we lift $f$ and $f_0$ to $F$ and $F_0$ 
by the exponential product $\pi (z,w) = (e^z, e^w)$; 
that is, 
$\pi \circ F = f \circ \pi$ and $\pi \circ F_0 = f_0 \circ \pi$.
More precisely,
we define
\[
F(Z, W) = (P(Z), Q(Z,W)) 
\]
\[
= (\delta Z + \log (1 + \zeta (e^Z)),
\gamma Z + dW + \log (1 + \eta (e^Z, e^W)))
\]
and $F_0 (Z,W) = (\delta Z, \gamma Z + d W)$.
By Lemma \ref{main lemma},
we may assume that
\[
\| F - F_0 \| < \tilde{\varepsilon} \text{ on } \pi^{-1} (U_r),
\]
where $||(Z,W)|| = \max \{ |Z|, |W| \}$ and $\tilde{\varepsilon} = \log (1 + \varepsilon)$.
Similarly, 
we can lift $\phi_n$ to $\Phi_n$ so that 
the equation $\Phi_n = F_0^{-n} \circ F^n$ holds; thus, for any $n \geq 1$, 
\[
\Phi_n (Z, W) = \left( \frac{1}{\delta^n} P_n(Z), 
\frac{1}{d^n} Q_n(Z,W) - \frac{\gamma_n}{\delta^n d^n} P_n(Z) \right), 
\] 
where $(P_n(Z),Q_n(Z,W)) = F^n(Z,W)$.
Let $\Phi_n = (\Phi_n^1, \Phi_n^2)$. 
Then
\[
|\Phi_{n+1}^1 - \Phi_n^1|
= \left| \frac{P_{n+1}}{\delta^{n+1}} - \frac{P_n}{\delta^n} \right|
= \frac{|P_{n+1} - \delta P_n|}{\delta^{n+1}} < \frac{1}{\delta^{n+1}} \tilde{\varepsilon}
\text{ and}
\]
\[
|\Phi_{n+1}^2 - \Phi_n^2|
= \left| \left\{ \frac{Q_{n+1}}{d^{n+1}} - 
\frac{\gamma_{n+1} P_{n+1}}{\delta^{n+1} d^{n+1}} \right\}
- \left\{ \frac{Q_n}{d^n} - 
\frac{\gamma_n P_n}{\delta^n d^n} \right\} \right|
\]
\[
= \left| \frac{Q_{n+1}}{d^{n+1}} 
- \frac{\gamma P_n}{d^{n+1}} - \frac{Q_n}{d^n} \right| 
+ \left| \frac{\gamma_{n+1} P_{n+1}}{\delta^{n+1} d^{n+1}}  
- \frac{\gamma_n P_n}{\delta^n d^n} - \frac{\gamma P_n}{d^{n+1}} \right|
\]
\[
= \frac{|Q_{n+1} - (\gamma P_n + d Q_n)|}{d^{n+1}} 
+ \frac{\gamma_{n+1} |P_{n+1} - \delta P_n|}{\delta^{n+1} d^{n+1}} 
< \frac{1}{d^{n+1}} \tilde{\varepsilon} + \frac{\gamma_{n+1}}{\delta^{n+1} d^{n+1}} \tilde{\varepsilon}.
\]
Hence $\Phi_n$ converges uniformly to $\Phi$.
In particular, 
\[
\| \Phi - id \| < \max \left\{ \frac{1}{\delta - 1}, 
\frac{1}{d-1} + \frac{\gamma}{\delta - d} 
\left( \frac{1}{d - 1} - \frac{1}{\delta - 1} \right) \right\} \tilde{\varepsilon}
\text{ if } \delta \neq d, \text{ and}
\]
\[
\| \Phi - id \| < \left\{ \frac{1}{d - 1} + \frac{\gamma}{(d-1)^2} \right\} \tilde{\varepsilon}
\text{ if } \delta = d.
\]
By the inequality $|e^{z_1}/e^{z_2} - 1| \leq |z_1 - z_2| e^{|z_1 - z_2|}$,
the uniform convergence of $\Phi_n$ translates into that of $\phi_n$.
Therefore,
$\phi$ is holomorphic on $U_r \setminus \{ zw = 0 \}$.
In particular, if $||\Phi - id|| < \varepsilon$, 
then $||\phi - id|| < \varepsilon e^{\varepsilon} ||id||$.
Hence 
$\phi \sim id$ on $U_r \setminus \{ zw = 0 \}$ as $r \to 0$.
Thanks to Riemann's extension theorem,
$\phi$ extends holomorphically to $U_r$,
and $\phi \sim id$ on $U_r^{}$ as $r \to 0$.

\subsection{Injectivity of $\phi$}

We prove that,
after shrinking $r$ if necessary,
the lift $F$ is injective on $\pi^{-1} (U_r)$.
Hence $F^n$, $\Phi_n$ and $\Phi$ are injective on the same region.
The injectivity of $\Phi$ derives that of $\phi$ because  $\Phi \sim id$.
 
It is enough to consider Case 4. 
In that case,
 $F$ is holomorphic on $V$, where 
\[
V 
= \pi^{-1} (U_{r})
= \left\{ (l_1 + l_2) \mathrm{Re} Z - l_2 \log r < \mathrm{Re} W < l_1 \mathrm{Re} Z + \log r \right\}.
\]
In particular,
$P$ is holomorphic and $|P- \delta Z| < \tilde{\varepsilon}$ on $H$,
where
\[
H = \left\{ Z \ | \ \mathrm{Re} Z < \left( 1 + l_2^{-1} \right) \log r \right\}.
\]
Rouch\'e's theorem guarantees the injectivity of $P$ on $H'$,
where
\[
H' =
\left\{ Z \ \Big| \ \mathrm{Re} Z < \left( 1 + \dfrac{1}{l_2} \right) \log r - \dfrac{2 \tilde{\varepsilon}}{\delta} \right\}
\subset H.
\]

\begin{prop}\label{biholo of lift}
The function
$P$ is injective on $H'$. 
\end{prop}

\begin{proof}
Let $Z_1$ and $Z_2$ be two points in $H'$
such that $P(Z_1) = P(Z_2)$,
and show that $Z_1 = Z_2$. 
Define $g (Z) = P(Z) - P(Z_1)$ and $h (Z) = \delta Z - P(Z_1)$.
Then $|g - h| = |P -\delta Z| < \tilde{\varepsilon}$ on $H$.
By the definitions of $H$ and $H'$,
there is a smooth, simply closed curve $\Gamma$ in $H$
whose distances from $Z_1$ and $Z_2$ are greater than $2 \tilde{\varepsilon}/ \delta$
and whose interior contains the two points $Z_1$ and $Z_2$.
Hence 
\[
|h| = |\delta Z - P(Z_1)| \geq |\delta Z - \delta Z_1| - |\delta Z_1 - P(Z_1)| 
    > 2 \tilde{\varepsilon} - \tilde{\varepsilon} = \tilde{\varepsilon}
\]
on $\Gamma$.
Therefore,
$|g - h| < |h|$ on $\Gamma$.
Rouch\'e's theorem implies that 
the number of zero points of $g$ is exactly one
in the region surrounded by $\Gamma$;
thus $Z_1 = Z_2$. 
\end{proof}

Let $V'_Z = V' \cap (\{ Z \} \times \mathbb{C})$, 
where
\[
V' =
\left\{ (l_1 + l_2) \mathrm{Re} Z - l_2 \log r + \frac{2 \tilde{\varepsilon}}{d} 
 < \mathrm{Re} W < l_1 \mathrm{Re} Z + \log r - \frac{2 \tilde{\varepsilon}}{d} \right\}
\subset V.
\]

The same argument induces the injectivity of $Q_Z$ on $V'_Z$. 

\begin{prop}\label{biholo of lift}
The function
$Q_Z$ is injective on $V'_Z$ for any fixed $Z$. 
\end{prop}

Note that $V' \subset \left\{ \mathrm{Re} Z < \left( 1 + \dfrac{1}{l_2} \right) \log r - \dfrac{4 \tilde{\varepsilon}}{l_2 d} \right\}$
and let $C = \max \left\{ \dfrac{1}{d}, \dfrac{l_2}{2 \delta} \right\}$.

\begin{cor}
The maps $F$, $F^n$, $\Phi_n$ and $\Phi$ are injective on
\[
\left\{ (l_1 + l_2) \mathrm{Re} Z - l_2 \log r + 2 C \tilde{\varepsilon}
 < \mathrm{Re} W < l_1 \mathrm{Re} Z + \log r - 2 C \tilde{\varepsilon} \right\}.
\]
\end{cor}

As mentioned above,
the injectivity of $\Phi$ derives that of $\phi$. 

\begin{prop}\label{biholo of phi}
The B\"{o}ttcher coordinate $\phi$ is injective on 
\[
\left\{ 
\frac{(1 + \varepsilon)^{2C}}{r^{l_2}} |z|^{l_1 + l_2} < |w| < \frac{r}{(1 + \varepsilon)^{2C}} |z|^{l_1}  
\right\}.
\]
\end{prop}

\begin{rem}\label{the critical set of f in U}
Since $f \sim f_0$ on $U_r$,
it follows that $Df \sim Df_0$ on $U_r$. 
Hence the intersection of the critical set $C_f$ of $f$ and $U_r$ is included in $\{ zw = 0\}$
for small $r$.
By almost the same arguments as in Section 9,
we can show the following:
$\phi$ extends to a biholomorphic map on $U_R$
if $U_R \cap C_f \subset \{ zw = 0\}$ and $U_R \subset A_f$,
where $A_f$ is the union of all the preimages of $U_r$ under $f$. 
\end{rem}

\section{The case $d = 1$}

We prove Lemma \ref{main lemma for d=1} and Theorem  \ref{main theorem for d=1}
in this section.
Let
\[
f(z,w) =
(z^{\delta} + O(z^{\delta + 1}), b z^{\gamma} w + \sum b_{ij} z^{i} w^{j}),
\]
where $b = b_{\gamma d} \neq 0$
and $\gamma \geq 1$.
The proof of the uniform convergence of $\phi_n$ is different from the case $d \geq 2$;
we use the same idea as in \cite{u2} to prove it and,
in addition,
we need a new number $M$ in Lemma \ref{lem2: d = 1} for Case 4.
Example \ref{example} shows that 
we can not remove the additional condition $\delta \neq T_k$ for any $k$.

The invariance of $U_r$ and so Lemma \ref{main lemma for d=1}
follow from the additional condition. 

\begin{proof}[Proof of Lemma \ref{main lemma for d=1}]
The proof of the former statement is the same as the case $d \geq 2$.
We show that the condition $\delta \neq T_k$ for any $k$
induces the invariance of $U_r$. 

For Case 2,
the condition implies that $\delta < T_{s-1}$,
which is equivalent to the inequality $\tilde{\gamma} = \gamma + l_1 d - l_1 \delta > 0$.
Hence $f$ preserves $U_r$ for small $r$.

For Case 3,
the condition implies that $T_1 < \delta$,
which is equivalent to the inequality $\delta > l_2^{-1} \gamma + d = \tilde{d}$.
Hence $f$ preserves $U_r$ for small $r$.

For Case 4,
the condition implies that $T_{k} < \delta < T_{k-1}$,
which is equivalent to the inequalities $\tilde{\gamma} > 0$ and $\delta > \tilde{d} > d$.
Hence $f$ preserves $U_r$ for small $r$.
\end{proof}

More strongly,
$f^n$ contracts $U_r$ rapidly, and
the following lemma is the beginning of the proof of the uniform convergence of $\phi_n$.

\begin{lem}\label{lem1: d = 1}
If $d = 1$ and $\delta \neq T_k$ for any $k$, 
then $f^n(U_r) \subset U_{r/2^n}$ for small $r$.
\end{lem}

\begin{proof}
It is enough to show the lemma for Case 4.
We first give an abstract idea of the proof.
If $b = 1$ then, formally,
\[
\tilde{f}_2 (t,c) \sim 
( t^{\delta - l_2^{-1} \tilde{\gamma}} c^{l_2^{-1} (\delta - \tilde{d})}, \ t^{\tilde{\gamma}} c^{\tilde{d}} )
\text{ on } 
\{ |t| < r, |c| < r \}.
\]
By assumption,
$\delta - l_2^{-1} \tilde{\gamma} > d = 1$,
$l_2^{-1} (\delta - \tilde{d}) > 0$,
$\tilde{\gamma} > 0$ and 
$\tilde{d} > d = 1$.
If $\tilde{f}_2$ is well-defined,
then the origin is superattracting,
and it is easy to check that
\[
\tilde{f}_2( \{ |t| < r, |c| < r \} ) \subset \{ |t| < r/2, |c| < r/2 \}
\text{ and so}
\]
\[
\tilde{f}_2^n( \{ |t| < r, |c| < r \} ) \subset \{ |t| < r/2^n, |c| < r/2^n \}.
\]

This idea provides a proof immediately.
Actually,
\[
\left| \frac{p(z)^{1 + l_1 l_2^{-1}}}{q(z,w)^{l_2^{-1}}} \right| 
< C_1 \left| t^{\delta - l_2^{-1} \tilde{\gamma}} c^{l_2^{-1} (\delta - \tilde{d})} \right|
< C_1 |t|^{\delta - l_2^{-1} \tilde{\gamma} - 1} \cdot |t|
< \dfrac{1}{2} \cdot r 
\text{ and}
\]
\[
\left| \frac{q(z,w)}{p(z)^{l_1}} \right| 
< C_2 \left| t^{\tilde{\gamma}} c^{\tilde{d}} \right|
< C_2 |c|^{\tilde{d} - 1} \cdot |c| 
< \dfrac{1}{2} \cdot r 
\]
for some constants $C_1$ and $C_2$ and for small $r$.
Hence 
\[
f(U_r) \subset U_{r/2}
\text{ and so }
f^n(U_r) \subset U_{r/2^n}.
\]
\end{proof}

Let $M = 1$ for Cases 1, 2 and 3,
and $M = \min \{ \min \{ \tilde{n}_j - \tilde{\gamma} : \tilde{n}_j > \tilde{\gamma} \}, 1 \}$ for Case 4,
where $ \tilde{n}_j = n_j + l_1 m_j - l_1 \delta$
and $\tilde{\gamma} = \gamma + l_1 d - l_1 \delta$.
Then $0 < M \leq 1$.

\begin{lem}\label{lem2: d = 1}
If $d = 1$ and $\delta \neq T_k$ for any $k$, 
then 
\[
|\zeta (p^n)| < C_1 \cdot \dfrac{r}{2^n} \text{ and } |\eta (f^n)| < C_2 \left( \dfrac{r}{2^n} \right)^M
\]
on $U_r$
for some constants $C_1$ and $C_2$.
\end{lem}

\begin{proof}
It is enough to consider Case 4
and show the later inequality.
Let $|w| = |z^{l_1}c|$ and $|z| = |tc^{l_2^{-1}}|$.
Then 
\[
|\eta|
= \left| \sum \frac{b_{ij} z^{i} w^{j}}{bz^{\gamma} w} \right| 
\leq \sum \frac{|b_{ij}|}{|b|} 
  |t|^{\tilde{i} - \tilde{\gamma}} |c|^{\tilde{j} - \tilde{d}},
\]
where 
$\tilde{i} \geq \tilde{\gamma}$ and $\tilde{j} \geq \tilde{d}$.
More precisely,
$\tilde{i} - \tilde{\gamma} \geq M$ if $\tilde{i} > \tilde{\gamma}$, and 
$\tilde{j} - \tilde{d} = j - d \geq 1$ if $\tilde{i} = \tilde{\gamma}$.
Hence there exist constants $A$ and $B$ such that
$|\eta| \leq A |t|^M + B |c|$
and so $|\eta| \leq A |t|^M + B |c|^M$.
It then follows from Lemma \ref{lem1: d = 1} that 
$|\eta (f^n)| < (A+B) (r/2^n)^M$ on $U_r$.
\end{proof}

Now we are ready to prove the uniform convergence of $\phi_n$.

\begin{prop}
If $d = 1$ and $\delta \neq T_k$ for any $k$,  
then $\phi_n$ converges uniformly to $\phi$ on $U_r$.
Moreover,
for any small $\varepsilon$, there is $r$ such that
$||\phi - id|| < \varepsilon ||id||$ on $U_r$.
\end{prop}

\begin{proof}
Let $\Phi_n$ be the lift of $\phi_n$ 
and $\Phi_n = (\Phi_n^1, \Phi_n^2)$ as in Section 6.
It is enough to show the uniform convergence of $\Phi_n^2$.
By Lemma \ref{lem2: d = 1},
\[
|\Phi_{n+1}^2 - \Phi_n^2|
\leq \frac{|Q(F^n) - Q_0(F^n)|}{d^{n+1}} 
+ \frac{\gamma_{n+1} |P(P^n) - P_0(P^n)|}{\delta^{n+1} d^{n+1}} 
\]
\[
\leq |\eta \circ \pi (F^n)| + \frac{\gamma}{\delta - 1} |\zeta \circ \pi (P^n)|
< \left( C_2 + \frac{\gamma}{\delta - 1} C_1 \right) \left( \frac{r}{2^n} \right)^M.
\]
\end{proof}

The injectivity of $\phi$ follows from the same proof as the case $d \geq 2$,
which completes the proof of Theorem  \ref{main theorem for d=1}.

\begin{prop}
If $d = 1$ and $\delta \neq T_k$ for any $k$, 
then $\phi$ is injective on 
\[
\left\{ 
\frac{(1 + \varepsilon)^{2C}}{r^{l_2}} |z|^{l_1 + l_2} < |w| < \frac{r}{(1 + \varepsilon)^{2C}} |z|^{l_1}  
\right\}.
\]
\end{prop}

Finally,
we exhibit an example that does not satisfies the additional condition. 

\begin{example}\label{example}
Let $f(z,w) = (z^2, z^{\gamma} w + z^{2 \gamma})$,
where $\gamma \geq 1$.
Then $\delta = T_1$, and
$f$ is semiconjugate to $g(z,w) = (z^2, w + 1)$ 
by $\pi (z,w) = (z, z^{\gamma} w): \pi \circ g = f \circ \pi$.
\end{example}

For this example,
Theorem \ref{main theorem for d=1} does not hold.
In fact,
if we had a B\"{o}ttcher coordinate that conjugates
$f$ to $f_0 (z,w) = (z^2, zw)$,
then $g$ should be conjugate to $g_0 (z,w) = (z^2, w)$.
However,
the translation $w \to w + 1$ can not be conjugate to the identity $w \to w$.
Also $f$ can not be conjugate to $f_0 (z,w) = (z^2, z^{2 \gamma})$, 
which is not dominant, on any open set.
This example is a generalization of Example 5.2 in \cite{u2} for the case $b = 1$.

\section{Uniqueness of B\"{o}ttcher coordinates}

In one dimension
the uniqueness of a B\"{o}ttcher coordinate is completely understood.
We obtained a similar result for polynomial skew products in \cite{u1}
with two suitable conditions.
The same argument works for Cases 1 and 2 if $d \geq 2$.

Let $p(z) = z^{\delta} + O(z^{\delta + 1})$,
a holomorphic germ with a superattracting fixed point at the origin,
and $p_0(z) = z^{\delta}$;
we assume that $a_{\delta} = 1$ for simplicity.
If we do not impose the condition $\varphi' (0) = 1$ on
the B\"{o}ttcher coordinate $\varphi$ for $p$,
then a B\"{o}ttcher coordinate $\varphi$ is unique 
up to multiplication by an $(\delta - 1)$st root of unity
as stated in \cite{m}.
Let $\varphi_1$ and $\varphi_2$ be conformal functions
that conjugate $p$ to $p_0$.
Then the composition $\varphi_2 \circ \varphi_1^{-1}$ conjugates $p_0$ to itself.
Hence we may assume that $p = p_0$
for the statement of the uniqueness of a change of coordinates.

\begin{lem}[\cite{m}]\label{uniq lem}
Let $\varphi$ be a conformal function defined on a neighborhood of the origin,
with $\varphi (0) = 0$,
that conjugates $p_0$ to itself.
Then $\varphi (z) = c_1 z$,
where $c_1^{\delta - 1} = 1$. 
\end{lem}

\begin{proof}
We rewrite the proof of Theorem 9.1 in \cite{m}.
Since $\varphi$ is holomorphic at the origin,
it has the Taylor expression
$\varphi (z) = c_0 + c_1 z + c_2 z^2 + \cdots$.
Note that $c_0 = 0$ and $c_1 \neq 0$
since $\varphi (0) = 0$ and $\varphi$ is conformal.
Therefore, 
\[
\varphi (z) = c_1 z + c_2 z^2 + \cdots,
\]
where $c_1 \neq 0$.
Let $\varphi (z) = c_1 z + c_k z^k + \cdots$ for an integer $k \geq 2$.
The identity $\varphi (z^{\delta}) = \varphi (z)^{\delta}$ 
then implies that $c_1^{\delta} = c_1$ and $c_k = 0$
because $\delta \geq 2$. 
\end{proof}

Note that
the condition $\varphi (0) = 0$ in the lemma
can be replaced by the stronger condition $|\varphi| \sim |z|$ as $z \to 0$,
and that
$\varphi (z) = z^n$ conjugates $p_0$ to itself
for any integer $n \geq 1$,
although it is not conformal.

Let $f_0 (z,w) = ( z^{\delta}, z^{\gamma} w^{d})$,
where $\delta \geq 2$, $\gamma \geq 0$,
$d \geq 0$ and $\gamma + d \geq 2$.
By weakening the condition $\phi \sim id$ 
to the condition $|\phi| \sim |id|$,
we can generalize the lemma above to the skew product case as follows.

\begin{prop}\label{uniq prop} 
Let $\phi$ be a biholomorphic map defined on $U$
that conjugates $f_0$ to itself,
where $U$ is an open set of the form of Case 1 or Case 2.
Assume that $d \geq 2$,
that $\phi$ is a skew product of the form
$\phi (z,w) = (\phi_1 (z), \phi_2 (z,w))$
and that $|\phi| \sim |id|$ on $U_r$ as $r \to 0$.
Then $\phi (z,w) = (c_1 z, c_2 w)$, 
where $c_1^{\delta - 1} = 1$ and $c_1^{\gamma} c_2^{d - 1} = 1$. 
\end{prop}

\begin{proof}
It follows from Lemma {\rmfamily \ref{uniq lem}}
that $\phi_1(z) = c_1 z$,
where $c_1^{\delta - 1} = 1$.
Because $U$ intersects the $z$-axis
for Cases 1 and 2,
$\phi_2$ is holomorphic at the origin in $w$ direction
for any fixed $z$,
and so it has the Taylor expression
$\phi_2 (z,w) = c_0(z) + c_1(z) w + c_2(z) w^2 + \cdots$
on the fiber.
Since $|\phi_2| \sim |w|$,
the ratio $|\phi_2 /w|$ should be bounded on $U$,
which implies that $c_0(z) \equiv 0$.
On the other hand, 
$c_1(z) \neq 0$
since $\phi_2$ is conformal at $w = 0$.
Therefore, 
\[
\phi_2 (z,w) = c_1(z) w + c_2(z) w^2 + \cdots,
\]
where $c_1(z) \neq 0$.
Let $\phi_2 (z,w) = c(z) w + c_k(z) w^k + \cdots$
for an integer $k \geq 2$.
The identity $\phi_2 \circ f_0 = \phi_1^{\gamma} \phi_2^{d}$ implies that
\[
c(z^{\delta}) z^{\gamma} w^d + c_k(z^{\delta}) z^{k \gamma} w^{kd} + \cdots
\]
\[
= (c_1 z)^{\gamma} \{ c(z)^d w^d + d c(z)^{d-1} c_k(z) w^{d-1 + k} + \cdots \}.
\]
Hence
$c(z^{\delta}) = c_1^{\gamma} c(z)^d$ and $c_k(z) \equiv 0$
because $d \geq 2$.
Let $c(z) = c_2 z^n + O(z^{n+1})$ for an integer $n \geq 0$.
The identity $c(z^{\delta}) = c_1^{\gamma} c(z)^d$
implies that $c(z) = c_2$ if $\delta \neq d$
and $c(z) = c_2 z^n$ if $\delta = d$,
where $c_1^{\gamma} c_2^{d - 1} = 1$.
Let $\phi_2 (z,w) = c_2 z^n w$.
Then $n = 0$
since $|\phi_2| \sim |w|$.
\end{proof}

\begin{rem}
If we replace the condition $|\phi| \sim |id|$ 
in the proposition to the condition that 
$\phi$ preserves the $z$-axis and $w$-axis, respectively,
then we have the other possibility:
for any integer $n \geq 0$,
the map $\phi (z,w) = (c_1 z, c_2 z^n w)$
is biholomorphic on $U$ for Case 2 and
conjugates $f_0$ to itself if $\delta = d$.
\end{rem}

Whereas
we can use the Taylor expression of $\phi_2$ on the fibers
for Cases 1 and 2
since $U$ intersects the $z$-axis,
we can only use the Laurent expression of $\phi_2$ 
for Cases 3 and 4
since $U$ is disjoint from the $z$-axis,
and so the same argument does not work.

\section{Extension of B\"{o}ttcher coordinates}

In one dimension 
there is a complete statement in dynamical viewpoint
on the extension of the B\"{o}ttcher coordinate $\varphi$
of a global holomorphic function $p$
with a superattracting fixed point at the origin;
see Theorem 9.3 in \cite{m}.
Roughly speaking,
$\varphi$ extends until it meets the other critical points of $p$ than the origin.
We obtained a similar statement  
for polynomial skew products in \cite{u1};
the B\"{o}ttcher coordinate near infinity for a polynomial skew product
extends until it meets the critical set of the polynomial map.
The situation in this paper is more or less different from that in \cite{u1};
the most major difference is that
we permit the critical set of $f$ to intersect $U$ and/or $V$, 
the region where the B\"{o}ttcher coordinate $\phi$ will be extended,
in the $z$-axis and $w$-axis. 
However,
the almost same arguments including analytic continuation work 
outside the $z$-axis and $w$-axis,
and we manage to obtain a similar result
thanks to Riemann's and Hartogs' extension theorems;
$\phi$ extends
until it meets the other critical set of $f$ than the $z$-axis and $w$-axis.

Let $f$ be defined globally in this section;
for example,
let $f$ be a holomorphic skew product
defined on $\{ |z| < R \} \times \mathbb{C}$
for large enough $R > 0$.
We assume that $f$ has a superattracting fixed point at the origin
and satisfies the conditions in Theorems \ref{main theorem} or \ref{main theorem for d=1}
so that it has the B\"{o}ttcher coordinate $\phi$ on $U$.
Let $\psi = (\psi_1, \psi_2)$ be the inverse of $\phi$.
Because $\phi \sim id$,
we may say that $\psi$ is biholomorphic on $U$.  
Our aim in this section is actually to extend $\psi$ from $U$ to a larger region $V$.

We first consider the dynamics of the monomial map $f_0$ in Section 9.1;
in particular,
we calculate the union $A_{f_0}$ of all the preimages of $U$ under $f_0$.
Then we provide a reasonable definition of $V$ in Section 9.2,
which is included in $A_{f_0} \cup \{ zw = 0 \}$.
Finally,
we state our result precisely and prove it in Section 9.3.

\subsection{Monomial maps}

Let $f_0 (z,w) = (z^{\delta}, z^{\gamma} w^{d})$,
where $\delta \geq 2$, $\gamma \geq 0$,
$d \geq 1$ and $\gamma + d \geq 2$;
we assume that the coefficients are both $1$ for simplicity.
It has a superattracting fixed point at the origin.

We first emphasize that 
the $z$-axis and $w$-axis are special curves in the following senses:
$(i)$ 
the critical set $C_{f_0}$ of $f_0$ is included in the $z$-axis and $w$-axis;
more precisely,
$C_{f_0} = \{ zw = 0 \}$ if $d \geq 2$, and
$C_{f_0} = \{ z = 0 \}$ if $d = 1$, and
$(ii)$ 
they are forward $f_0$-invariant,
respectively.
In particular, $f_0$ is rigid.

Next we calculate the union of the all preimages of $U$ under $f_0$.
Let
\[
A_{f_0} = A_{f_0} (U) = \bigcup_{n \geq 0} f_0^{-n} (U),
\]
which is included in the attracting basin of the origin for $f_0$.
The affine function 
\[
R(a) = \frac{\delta a - \gamma}{d}
\]
plays a central role to calculate $A_{f_0}$.
Note that $f_0^{-n} (U)$ is equal to
\begin{enumerate}
\item $\{ |z| < r^{{1/\delta}^n}, |w| < r^{1/d^{n}} |z|^{R^n(0)} \}$ for Case 1,
\item $\{ |z| < r^{1/{\delta}^{n}}, 0 < |w| < r^{1/d^{n}} |z|^{R^n(l_1)} \}$ for Case 2,
\item $\{ r^{- l_2 /d^{n}} |z|^{R^n(l_2)} < |w| < r^{1/d^{n}} |z|^{R^n(0)} \}$ for Case 3, and
\item $\{ 0 < r^{- l_2 /d^{n}} |z|^{R^n(l_1 + l_2)} < |w| < r^{1/d^{n}} |z|^{R^n(l_1)} \}$ for Case 4.
\end{enumerate}
If $\delta \neq d$,
then 
\[
R(a) = \frac{\delta}{d} (a - \alpha_0) + \alpha_0
\text{ and so }
R^n (a) = \left( \frac{\delta}{d} \right)^n (a - \alpha_0) + \alpha_0,
\]
where $\alpha_0 = \gamma /(\delta - d)$.
Therefore,
for Case 1, the set $A_{f_0}$ is equal to
\begin{enumerate}
\item[(i)] $\{ |z| < 1 \}$ if $\delta \geq d$ and $\gamma > 0$,
\item[(ii)] $\{ |z| < 1, |w| < |z|^{\alpha_0} \}$ if $\delta < d$ and $\gamma > 0$, where $\alpha_0 < 0$, or
\item[(iii)] $\{ |z| < 1, |w| < 1 \}$ if $\gamma = 0$.
\end{enumerate}
For Case 2, the inequalities $T_s \geq \delta$ and $\gamma > 0$ hold and $A_{f_0}$ is equal to
\begin{enumerate}
\item[(i)] $\{ |z| < 1, |w| < |z|^{l_1} \}$ if $T_s = \delta > d \geq 2$, 
\item[(ii)] $\{ |z| < 1, |w| < r|z|^{l_1} \}$ if $T_s = \delta > d = 1$, 
\item[(iii)] $\{ 0 < |z| < 1 \}$ if $T_s > \delta \geq d$, or
\item[(iv)] $\{ 0 < |z| < 1, |w| < |z|^{\alpha_0} \}$ if $\delta < d$, where $\alpha_0 < 0$.
\end{enumerate}
For Case 3, the inequalities $\delta \geq T_1 \geq d$ hold and $A_{f_0}$ is equal to
\begin{enumerate}
\item[(i)] $\{ |z| < 1, w \neq 0 \}$ if $\delta > d$ and $\gamma > 0$,
\item[(ii)] $\{ |z| < 1, 0 < |w| < 1 $ if $\delta > d$ and $\gamma = 0$, or
\item[(iii)] $\{ |z|^{l_2} < |w| < 1 \}$ if $\delta = d$ (and $\gamma = 0$).
\end{enumerate}
For Case 4, 
the inequalities $T_{k-1} \geq \delta \geq T_k > d$ and $\gamma > 0$ hold and $A_{f_0}$ is 
\begin{enumerate}
\item[(i)] $\{ |z| < 1, 0 < |w| < |z|^{l_1} \}$ if $T_{k-1} = \delta > T_k > d \geq 2$, 
\item[(ii)] $\{ |z| < 1, 0 < |w| < r|z|^{l_1} \}$ if $T_{k-1} = \delta > T_k > d = 1$, 
\item[(iii)] $\{ 0 < |z| < 1, w \neq 0 \}$ if $T_{k-1} > \delta > T_k$, 
\item[(iv)] $\{ 0 < |z| < 1, |z|^{l_1 + l_2} < |w| \}$ if $T_{k-1} > \delta = T_k > d \geq 2$, or
\item[(v)] $\{ 0 < |z| < 1, r^{-l_2} |z|^{l_1 + l_2} < |w| \}$ if $T_{k-1} > \delta = T_k > d = 1$.
\end{enumerate}
Note that
$A_{f_0}$ does not intersect the $z$-axis and/or $w$-axis
but $\overline{A_{f_0}}$ does
for many cases,
and that
$A_{f_0} \subset int \overline{A_{f_0}} \subset A_{f_0} \cup \{ zw = 0 \}$
for all cases.

\subsection{Definition of $V$}

We require the region $V$ to be simply connected, Reinhardt domain, 
and included in $A_{f_0} \cup \{ zw = 0 \}$.
More specifically,
we define
\[
V = \{ r_2^{-1} |z|^{a_2} < |w| < r_1 |z|^{a_1} \},
\]
where $r \leq r_1 \leq 1$, $r_2 \leq 1$ and
$- \infty \leq a_1 \leq l_1 \leq l_1 + l_2 \leq a_2 \leq \infty$.
We assume that
$U \subset V \subset int \overline{A_{f_0}}$. 
Then $V \setminus \{ zw = 0 \} \subset A_{f_0}$
and hence 
we use analytic continuation
outside the $z$-axis and $w$-axis.
The region $V$ realizes all the types of $int \overline{A_{f_0}}$
for suitable choices of the parameters $r_1$, $r_2$, $a_1$ and $a_2$.

\begin{rem}
It might seem to be natural to define
\[
V = \{ r_2^{-a_2} |z|^{a_1 + a_2} < |w| < r_1 |z|^{a_1} \},
\]
where $r \leq r_1 \leq 1$, $r \leq r_2 \leq 1$,
$- \infty \leq a_1 \leq l_1$ and
$l_2 \leq a_2 \leq \infty$.
However, 
if $a_1 = - \infty$ and $a_2 = \infty$,
then we can not compute $a_1 + a_2$.
In particular,
we have to set $a_1 = - \infty$ and $a_1 + a_2 = l_1 + l_2$
to realize $\{ |z| < 1, |z|^{l_1 + l_2} < |w| \}$,
the union of $A_{f_0}$ and the $w$-axis for Case 4 when $\delta = T_k$,
by $V$.
\end{rem}

\begin{rem}
We do not need to assume that $V \subset A_{f_0}$
because we show in the next subsection that,
if $\psi$ extends to a biholomorphic map on $V \setminus \{ zw = 0 \}$,
then it extends to a biholomorphic map on $V$.
This is different from \cite{u1};
the corresponding region $V_r^a$ in \cite{u1} is included in
the union $A_{f_0}^{\alpha}$ of all the preimages of the given open set $V_R$. 
Moreover,
we have two parameters $a_1$ and $a_2$ of weights in the definition of $V$,
whereas we needed only one parameter $a$ in \cite{u1}.
\end{rem}

\subsection{Statement and Proof}

As mentioned in Section 9.1,
the critical set $C_{f_0}$ of $f_0$ is included in the $z$-axis and $w$-axis.
Hence
$C_{f_0}$ may intersect $U$ and/or $V$,
unlike the situation in \cite{u1}.
In that case, 
$f$ is also expected to have the critical set $C_f$ in the $z$-axis and $w$-axis;
in fact, 
$\phi$ and so $\psi$ preserve the $z$-axis and/or $w$-axis
if they are defined there,
and $f$ has the critical set there.
Therefore,
we permit $C_f$ to intersect $U$ and/or $V$
in the $z$-axis and $w$-axis,
and use analytic continuation
outside the $z$-axis and $w$-axis
even for the case $V \subset A_{f_0}$.
Let 
\[
A_f = A_f (
U) = \bigcup_{n \geq 0} f^{-n} (U),
\]
which is included in the attracting basin of the origin for $f$.
Let $|\phi| = (|\phi_1|, |\phi_2|)$, 
which extends to a continuous map from $A_f$ to $\mathbb{R}^{2}$
via $(f_0 |_{\mathbb{R}^{2}})^{-n} \circ |\phi| \circ f^n$.

\begin{thm}
Let $f$ be defined globally. 
If $f$ has no critical points in $|\phi|^{-1} (V \cap \mathbb{R}_{>0}^{2})$,
then $\psi$ extends by analytic continuation
to a biholomorphic map on $V$.
\end{thm}

\begin{proof} 
We first show that 
$\psi$ extends to a holomorphic map on $V \setminus \{ zw = 0 \}$ by analytic continuation.
Let $\mathcal{U} = U \setminus \{ zw = 0 \}$ 
and $\mathcal{V} = V \setminus \{ zw = 0 \}$.
Take any points $x$ in $\mathcal{U}$ and $y$ in $\mathcal{V}$.
Connect $x$ and $y$ by a path $\Gamma$ in $\mathcal{V}$.
Since $\mathcal{V} \subset A_{f_0} (U)$,
there is an integer $n$ such that
$f_0^{n}$ maps $\Gamma$ into $U$.
In fact,
$f_0^{n} (\Gamma) \subset \mathcal{U}$.
Since $f$ has no critical points 
in $|\phi|^{-1} (\mathcal{V} \cap \mathbb{R}^{2})$, 
for any point in $\psi (f_0^{n} (\Gamma))$,
there is an open neighborhood $N$ in $\psi (\mathcal{U})$ such that
each branch of $f^{-n}$ from $N$ 
into $|\phi|^{-1} (\mathcal{V} \cap \mathbb{R}^{2})$ is well-defined. 
Because $\psi (f_0^{n} (\Gamma))$ is compact,
it is covered by finitely many open neighborhoods.
Pulling back this finite open covering by $\psi$ and $f_0^{n}$, 
we can take a finite open covering $\{ N_j \}_{j = 0}^{s}$ of $\Gamma$,
where $N_0$ and $N_s$ contain $x$ and $y$ respectively,
such that 
any branch of $f^{-n}$ on $\psi (f_0^{n} (N_j))$ 
is well-defined for any $j$.
We show that $\psi$ extends along $\Gamma$ 
by defining $f^{-n} \circ \psi \circ f_0^{n}$ approximately.
We may assume that $N_0 \subset \mathcal{U}$,
and we can define $f^{-n} \circ \psi \circ f_0^{n}$ on $N_0$ as $\psi$.
We next choose the branch of $f^{-n}$ on $\psi (f_0^{n} (N_1))$
such that $f^{-n} \circ \psi \circ f_0^{n}$ coincides with $\psi$ on $N_0 \cap N_1$.
Continuing this construction,
we can define a holomorphic map $f^{-n} \circ \psi \circ f_0^{n}$ on $N_j$
for any $j$ inductively;
thus we get the analytic continuation of $\psi$ along the path $\Gamma$.
Although $\mathcal{V}$ is not simply connected,
this analytic continuation does not depend on   
the base point $x$ and the path $\Gamma$
because $\psi$ is already defined on $U$.
Let us explain more precisely why it is independent of the choice of the path. 
Let $\Gamma_1$ and $\Gamma_2$ be two paths in $\mathcal{V}$ connecting $x$ and $y$.
Then the loop $\Gamma_2 \circ \Gamma_1^{-1}$ moves continuously in $\mathcal{V}$ to a loop in $\mathcal{U}$.
Therefore,
the analytic continuations along $\Gamma_1$ and $\Gamma_2$ 
have to coincide at the point $y$
because $\psi$ is already defined on $\mathcal{U}$.

Next we show that $\psi$ is homeomorphism on $\mathcal{V}$.
By the constriction of $\psi$, 
it is locally one-to-one,
and the set of all pairs $x_1 = (z_1,w_1) \neq x_2 = (z_2,w_2)$
with $\psi (x_1) = \psi (x_2)$ forms a closed subset of $\mathcal{V} \times \mathcal{V}$.
If $\psi (x_1) = \psi (x_2)$,
then $|z_1|=|z_2|$ and $|w_1|=|w_2|$ 
because $|\phi \circ \psi|=|id|$.
Assuming that there were such a pair with $\psi (x_1) = \psi (x_2)$,
we derive a contradiction.
There are two cases:
the minimum of $|z_1|$ exists or not.
First,
assume that the minimum exists,
which is positive.
Since $\psi$ is an open map,
for any $x'_1$ sufficiently close to $x_1$ 
we can choose $x'_2$ close to $x_2$ 
with $\psi (x'_1) = \psi (x'_2)$.
In particular,
we can choose $x'_j$ with $|z'_j| < |z_j|$,
which contradicts the choice of $z_j$.
Next,
assume that the minimum does not exist.
Then there is a pair with $0 < |z_1| = |z_2| < r$.
Fix such $z_1$.
For Cases 1 and 2,
the intersection of $\mathcal{V} \setminus \mathcal{U}$ and the fiber at $z_1$ is an annulus, 
and we can choose $|w_1|$ as minimal.
Using the same argument as above to the fibers,
we can choose $x'_1 = (z_1,w'_1)$ and $x'_2 = (z_2,w'_2)$
so that $\psi (x'_1) = \psi (x'_2)$ and $|w'_j| < |w_j|$,
which contradicts the choice of $w_j$.
For Cases 3 and 4,
the intersection may consist of two annuluses.
For this case,
we can choose $|w_1|$ as minimal in the outer annulus
or as maiximal in the inner annulus,
which contradicts the choice of $w_j$
by the same argument as above.

Finally,
we show that $\psi$ extends to a biholomorphic map on $V$.
It is well known that $\psi_1$ is well-defined and holomorphic at $z = 0$
and, more precisely,
$\psi_1 (0) = 0$.
We want to show that
$\psi_2$ extends holomorphically from $\mathcal{V}$ to $V$.
Then, clearly,
$\psi$ is biholomorphic on $V$
since it is biholomorphic on $\mathcal{V}$.
Case 1 is rather easy;
since $\psi_2$ is holomorphic on $U \cup \mathcal{V}$,
where $U$ is a neighborhood of the origin, 
it extends to a holomorphic map on $V$
thanks to Hartogs' extension theorem.
The other cases need another argument
since $U$ is not a neighborhood of the origin.
For Case 3,
assuming that $V$ intersects the $z$-axis,
we show that $\psi_2$ is bounded on $\mathcal{N}$, 
where $\mathcal{N} = N \setminus \{ w = 0 \}$
and $N = \{ |z|<r^{1 + l_2^{-1}}, |w|<r \}$.
Note that $U \subset \mathcal{N} \subset N \subset V$. 
Fix $z_0$ such that $|z_0| < r^{1 + l_2^{-2}}$ and
let $\mathbb{C}_{z_0} = \{ z_0 \} \times \mathbb{C}$.
Define $h (w) = \psi_2 (z_0, w)$;
that is,
we restrict the map to the vertical fiber.
Because $h \sim w$ on $U \cap \mathbb{C}_{z_0}$
and $h$ is homeomorphism on $\mathcal{N} \cap \mathbb{C}_{z_0}$,
it follows that the image under $h$
of the punctured disk $\mathcal{N} \cap \mathbb{C}_{z_0} = \{ (z_0, w) : 0<|w|<r\}$ is
surrounded by the image under $h$ 
of the outer boundary $\{ |w| = r \}$.
Therefore,
$\psi_2$ is bounded on $\mathcal{N}$.
Thanks to Riemann's extension theorem, 
$\psi_2$ extends to a holomorphic map from $\mathcal{N}$ to $N$.
Thanks to Hartogs' extension theorem,
$\psi_2$ extends to a holomorphic map from $N \cup \mathcal{V}$ to $V$.
Similar arguments hold for Cases 2 and 4.
For Case 2,
assuming that $V$ intersects the $w$-axis,
we can show that $\psi_2$ is bounded on $\{ |z|<r, |w|<r^{1 + l_1} \} \setminus \{ z = 0 \}$.
To show it,
we restrict $\psi_2$ to the horizontal lines.
For Case 4,
if $V$ intersects the $z$-axis and $w$-axis,
then we can show that $\psi_2$ is bounded on 
$\{ |z| < r^{1 + l_2^{-1}}, |w| < r^{1 + l_1 (1 + l_2^{-1})} \} \setminus \{ zw = 0 \}$.
To show it,
we consider the restrictions of $\psi_2$ to both the vertical fibers and horizontal lines.
\end{proof}
 
It follows from the construction that
$\psi (\mathcal{V}) \subset A_f (U)$. 
If one can prove that 
$\psi_2$ is bounded on $\mathcal{V} \cap K$
for any compact set $K$,
then it is clear that $\psi_2$ extends holomorphically
from $\mathcal{V}$ to $V$
thanks to Riemann's extension theorem. 

\begin{rem}
The open set, on which we constructed 
the B\"{o}ttcher coordinate near infinity for a polynomial skew product $f$ in \cite{u1},
is disjoint from the critical set of  $f$.
However,
the open set relates to the critical set of the rational extension of $f$ as follows.
We can extend $f$ to the rational map on a weighted projective space.
Then the line at infinity is included in the critical set of the rational map,
and intersects the closure of the open set.
\end{rem}


\end{document}